\documentclass[leqno,11pt,a4paper]{amsart}
\usepackage{cite}
\usepackage[margin=2.5cm]{geometry}
\usepackage[usenames,dvipsnames]{color}
\usepackage{hyperref}
\hypersetup{
    colorlinks,
    linkcolor={red!50!black},
    citecolor={blue!50!black},
    urlcolor={blue!80!black}
}
\usepackage{tikz}
\usepackage{mleftright}
\usepackage{booktabs}
\usepackage{colortbl}

\newcommand{\C}{\mathbb{C}}
\renewcommand{\P}{\mathbb{P}}
\newcommand{\Q}{\mathbb{Q}}
\newcommand{\Z}{\mathbb{Z}}
\newcommand{\N}{\Z_{\geq 0}}
\newcommand{\cO}{\mathcal{O}}
\newcommand{\sump}{{\Sigma^+}}
\newcommand{\sumn}{{\Sigma^-}}
\newcommand{\mmax}{{m_\mathrm{max}}}
\newcommand{\A}[1]{{\mathrm{A}_{#1}}}
\newcommand{\D}[1]{{\mathrm{D}_{#1}}}
\newcommand{\E}[1]{{\mathrm{E}_{#1}}}
\newcommand{\abs}[1]{\left\vert{#1}\right\vert}
\renewcommand{\gcd}[1]{\mathrm{gcd}\mleft\{{#1}\mright\}}
\renewcommand{\min}[1]{\mathrm{min}\mleft\{{#1}\mright\}}
\renewcommand{\max}[1]{\mathrm{max}\mleft\{{#1}\mright\}}
\newcommand{\cone}[1]{\mathrm{cone}\mleft\{{#1}\mright\}}
\renewcommand{\mod}{\ \mathrm{mod}\ }
\DeclareMathOperator{\Proj}{Proj}
\newcommand{\centerframebox}[1]{\fbox{\begin{minipage}{\dimexpr\textwidth-2\fboxsep-2\fboxrule\relax}\begin{center}#1\end{center}\end{minipage}}}
\newcommand{\ForAll}{\text{ for all }}
\newcommand{\ForAllEven}{\text{ for all even }}
\newcommand{\ForAllOdd}{\text{ for all odd }}
\newcommand{\padding}{\rule[-1.45ex]{0pt}{0.2em}}
\newcommand{\oddrow}{\rowcolor[gray]{0.95}\padding}
\newcommand{\evnrow}{\padding}
\newcolumntype{g}{>{\columncolor[gray]{0.95}\centering\arraybackslash}m{1.5em}}
\newcolumntype{y}{>{\columncolor[gray]{0.95}\centering\arraybackslash}m{2.3em}}
\newcolumntype{G}{>{\columncolor[gray]{0.95}\centering\arraybackslash}m{3.2em}}
\newcolumntype{h}{>{\columncolor[gray]{0.95}\centering\arraybackslash}m{3.9em}}
\newcolumntype{w}{>{\centering\arraybackslash}m{1.5em}}
\newcolumntype{x}{>{\centering\arraybackslash}m{2.3em}}
\newcolumntype{W}{>{\centering\arraybackslash}m{3.2em}}
\newcolumntype{H}{>{\centering\arraybackslash}m{3.9em}}
\newtheorem{thm}{Theorem}
\newtheorem{cor}[thm]{Corollary}
\newtheorem{lem}[thm]{Lemma}

\theoremstyle{definition}

\theoremstyle{remark}
\newtheorem{rem}[thm]{Remark}
\theoremstyle{remark}
\newtheorem{eg}[thm]{Example}
\numberwithin{equation}{section}
\numberwithin{thm}{section}
\begin{document}
\author[G.~Brown]{Gavin~Brown}
\address{Mathematics Institute\\University of Warwick\\Coventry\\CV4 7AL\\UK
\newline\indent
School of Mathematics\\Loughborough University\\Loughborough\\LE11 3TU\\UK}
\email{G.D.Brown@lboro.ac.uk} 
\author[A.~M.~Kasprzyk]{Alexander~Kasprzyk}
\address{Department of Mathematics\\Imperial College London\\London\\SW7 2AZ\\UK}
\email{a.m.kasprzyk@imperial.ac.uk} 
\title{Four-dimensional projective orbifold hypersurfaces}
\maketitle
\begin{abstract}
We classify four-dimensional quasismooth weighted hypersurfaces with small canonical class,
and verify a conjecture of Johnson and Koll\'ar on infinite series of quasismooth hypersurfaces with anticanonical hyperplane section in the case of fourfolds.
By considering the quotient singularities that arise, we classify those weighted hypersurfaces that are canonical, Calabi--Yau, and Fano fourfolds.
We also consider other classes of hypersurfaces, including Fano hypersurfaces of index greater than 1 in dimensions~3 and~4.
\end{abstract}
\section{Introduction}
A hypersurface $X\colon(F = 0)\subset w\P^{s-1}=\P(a_1,\dots,a_s)$ in weighted projective space is \emph{quasismooth} if its affine cone $C(X)\subset\C^{s}$ is nonsingular away from the origin. In this case $X$ is an orbifold. We use the notation $X_d\subset\P(a_1,\dots,a_s)$ to denote a general member of the linear system $\abs{\cO(d)}$ on $\P(a_1,\dots,a_s)$, and we refer to $X_d$ as a single variety, even though it represents the whole deformation family.

Without loss of generality we may assume that $X$ is well-formed: that is, $X$ does not intersect nontrivial orbifold strata of $w\P^{s-1}$ in a codimension~one locus; see Iano--Fletcher~\cite[6.9--10]{IF} for the divisibility conditions this imposes on $a_1,\dots,a_s,d$. It follows by Dolgachev~\cite[Theorem~3.3.4]{dolgachev} and~\cite[6.14]{IF} that $\omega_X=\cO_X(d-\sum a_i)$.

In this paper we classify quasismooth hypersurfaces of dimension at most four with small $\omega_X$; that is, with $\omega_X=\cO_X(k)$ for $k$ close to zero. Our main results are for fourfolds in the three cases $k=1$, $0$, and $-1$, which we summarise in, respectively, Theorems~\ref{thm!gt4},~\ref{thm!cy4}, and~\ref{thm!fano4} below. Within these classifications we identify the finite collections of varieties that satisfy additional Mori-theoretic hypotheses on singularities. We summarise a range of other results very briefly in Tables~\ref{tab!2d},~\ref{tab!3d}, and~\ref{tab!4d} in \S\ref{sec!results}; further details are available on the Graded Ring Database~\cite{grdb}.

These classifications are the result of a terminating algorithm, essentially following Johnson--Koll\'ar~\cite{JK,JK2} and Reid~\cite{C3f}. This algorithm imposes linear conditions on the integer sequence $(a_1,\dots,a_s,d)$, as we explain in~\S\ref{sec!alg} below. We demonstrate the essential idea in a beautiful small example in~\S\ref{sec!ADE} which recovers the ADE singularities as cones on quasismooth rational curve hypersurfaces; this is an elementary analogue of Cheltsov--Shramov's~\cite{CS} application of Yau--Yu's~\cite{YY} classification of isolated canonical hypersurface singularities. We have an implementation of the algorithm, available from~\cite{grdb}, for the computer algebra system Magma. This algorithm works for any values of dimension and canonical degree~$k$, although since it merely imposes necessary conditions, output that includes infinite series of solutions requires additional analysis. To illustrate this, in~\S\ref{sec!results} we recover the classification of low index del~Pezzo surfaces as a model case -- following~\cite{CS} and Boyer--Galicki--Nakamaye~\cite{BGN} -- as well as the index~two Fano threefolds.

We use the standard notation of~\cite{IF}. In particular, $\frac{1}{r}(a_1,\dots,a_m)$ denotes the germ of the quotient singularity $\C^m/\mu_r$, where $\mu_r$ acts with weights $x_i\mapsto\varepsilon^{a_i}x_i$. We work over $\C$ throughout. 

\subsection{Canonical fourfold hypersurfaces}\label{s!c4f}
A \emph{canonical fourfold} is a four-dimensional variety $X$ with $K_X$ ample and with $X$ possessing at worst canonical singularities. The variety $X$ embeds naturally in weighted projective space via its canonical ring
\[
X=\Proj R(X)\qquad\text{ where }\qquad R(X)=\bigoplus_{m\in\N}H^0\mleft(X,mK_X\mright).
\]
We classify the cases where this embedding is a quasismooth weighted hypersurface or, equivalently, where $X$ is an orbifold and $R(X)$ is minimally generated by six homogeneous generators.

The nonsingular fourfold $X_7\subset\P^5$ is the first example: $R(X_7)$ is generated in degree~one since $-K_{X_7}$ is very ample. More typically $-K_X$ will fail to be very ample, and so $R(X)$ will need generators in higher degree and the anticanonical embedding lies in projective space weighted in those degrees: for example $X_8\subset\P(2,1,1,1,1,1)$; or $X_{10}\subset\P(3,2,1,1,1,1)$, which has a single isolated terminal quotient singularity $\frac{1}{3}(1,1,1,2)$. Johnson--Koll\'ar~\cite[Cor~4.3]{JK2} prove that in any dimension that there are only finitely many cases of quasismooth hypersurfaces with $\omega_X=\cO(k)$, for each value of $k>0$.

\begin{thm}[Canonical fourfold hypersurfaces]\label{thm!gt4}
There are $1\,338\,926$ deformation families of well-formed, quasismooth four-dimensional hypersurfaces $X_d\subset\P(a_1,\dots,a_6)$ with $\omega_X=\cO(1)$. Of these, $649$ have a member with canonical singularities, and the general member in each of these cases has terminal singularities.
\end{thm}

Canonical fourfolds have two natural invariants: their (geometric) genus $p_g=h^0\mleft(X,K_X\mright)$ and their degree $K_X^4$. Amongst the $649$ canonical fourfolds of Theorem~\ref{thm!gt4}, $X_7\subset\P^5$ has the largest degree and
\[
X_{165}\subset\P(55,37,33,17,12,10)
\]
the smallest, with $K_{X_{165}}^4=1/830\,280$. Genera are distributed amongst the $649$ cases as shown in Table~\ref{tab:gt4_genus}. The number of cases with $h^0\mleft(X,mK_X\mright)=0$ for $m< b$ and $h^0\mleft(X,bK_X\mright)\neq 0$ (in other words, those where $b$ is the smallest weight) are given in Table~\ref{tab:gt4_b}. The extreme case is $X_{105}\subset\P(23,21,18,15,14,13)$, which has degree $1/226\,044$.
\begin{table}[ht]
\caption{The number of canonical fourfold hypersurfaces of genus $p_g$.}\label{tab:gt4_genus}
\centering
\begin{tabular}{rgwgwgwg}
\toprule
$p_g$&$0$&$1$&$2$&$3$&$4$&$5$&$6$\\
\#&$451$&$148$&$31$&$10$&$5$&$3$&$1$\\
\bottomrule
\end{tabular}
\end{table}

Chen--Chen~\cite[Theorem~8.2]{CC3} prove that $mK_X$ is birational for a canonical threefold with $p_g\ge2$ whenever $m\ge35$. Three of the low degree cases from the $649$ show that the requirement on $p_g$ is sharp:
\[
X_{72}\subset\P(36, 11, 9, 8, 6, 1),\quad
X_{78}\subset\P(39, 13, 10, 8, 6, 1),\quad\text{ and }\quad
X_{78}\subset\P(39, 14, 9, 8, 6, 1).
\]
\begin{table}[ht]
\caption{The number of canonical fourfold hypersurfaces with smallest weight $b$.}\label{tab:gt4_b}
\centering
\begin{tabular}{rgwgwgwgwgwgwg}
\toprule
$b$&$1$&$2$&$3$&$4$&$5$&$6$&$7$&$8$&$9$&$10$&$11$&$12$&$13$\\
\#&$198$&$175$&$135$&$48$&$48$&$10$&$19$&$7$&$4$&$2$&$2$&$0$&$1$\\
\bottomrule
\end{tabular}
\end{table}

\subsection{Calabi--Yau fourfold hypersurfaces}\label{s!cy4}
A \emph{Calabi--Yau fourfold} is a four-dimensional variety $X$ with $\omega_X=\cO_X$ and with $X$ possessing at worst canonical singularities, and satisfying the regularity conditions $\dim{H^1\mleft(X,\cO_X\mright)}=\dim{H^2\mleft(X,\cO_X\mright)}=0$.

In three dimensions, Calabi--Yau threefolds are well known from the Kreuzer--Skarke list~\cite{KS} of $473\,800\,776$ toric hypersurfaces. Of these $184\,026$ are hypersurfaces in weighted projective space. The Kreuzer--Skarke classification does not require that the hypersurface is quasismooth, however the \textsc{Palp} webpage~\cite{palp} allows one to extract that sublist: there are $7555$ quasismooth Calabi--Yau threefold hypersurfaces in weighted projective space.

We generate Calabi--Yau fourfolds as the case $s=6$, $k=0$ -- the additional regularity and singularity conditions are satisfied in all cases. Johnson--Koll\'ar~\cite[Theorem~4.1]{JK2} prove that, in any dimension, there are only finitely many quasismooth hypersurfaces with $k=0$.

\begin{thm}[Calabi--Yau fourfold hypersurfaces]\label{thm!cy4}
The are $1\,100\,055$ deformation families of well-formed, quasismooth hypersurfaces $X_d\subset\P(a_1,\dots,a_6)$ with $\omega_X=\cO_X$. In each case the general member has canonical singularities and is a Calabi--Yau fourfold.
\end{thm}

These Calabi--Yau fourfolds are polarised by an ample divisor $A\in\abs{\cO_X(1)}$. Their degree is the rational number $A^4$. The hypersurface with highest degree is the sextic $X_6\subset\P^5$, whilst the lowest degree $A^4=1/500\,625\,433\,457\,614\,850\,966\,280$ ($\sim 2\times 10^{-24}$) is achieved by
\[
X_{6521466}\subset\P(3260733, 2173822, 931638, 151662, 1806, 1805).
\]
Unsurprisingly, this case also has the largest ambient weight and the largest equation degree amongst all hypersurfaces.
\begin{table}[ht]
\caption{The number of Calabi--Yau fourfold hypersurfaces with $P_1=h^0\mleft(X,A\mright)$.}\label{tab:cy4_P1}
\centering
\begin{tabular}{rGWGWGWG}
\toprule
$P_1$&$0$&$1$&$2$&$3$&$4$&$5$&$6$\\
\#&$987\,884$&$109\,443$&$2576$&$134$&$14$&$3$&$1$\\
\bottomrule
\end{tabular}
\end{table}

The distribution of the hypersurfaces partitioned by $P_1=h^0\mleft(X,A\mright)$ is given in Table~\ref{tab:cy4_P1}. Table~\ref{tab:cy4_b} contains the number of hypersurfaces with $h^0(X,bA)\neq 0$, $b$ as small as possible. These are collected into ranges $500i+1\le b \le 500(i+1)$. The three largest minimum weights are are obtained by
\begin{align*}
X_{562500}&\subset\P(281250, 187500, 79619, 5167, 4500, 4464),\\
X_{594762}&\subset\P(297381, 198254, 84966, 4962, 4915, 4284),\\
\text{ and }X_{656250}&\subset\P(328125, 218750, 93750, 5250, 5208, 5167).
\end{align*}
The four extreme examples considered so far are double covers $x_1^2 = f(x_2,\dots,x_5)$. In total, $360\,346$ of the hypersurfaces are double covers.
\begin{table}[ht]
\caption{The number of Calabi--Yau fourfold hypersurfaces with smallest weight $b$, where $500i+1\le b \le 500(i+1)$.}\label{tab:cy4_b}
\centering
\begin{tabular}{rhxgwgwgwgwgwg}
\toprule
$i$&$0$&$1$&$2$&$3$&$4$&$5$&$6$&$7$&$8$&$9$&$10$\\
\#&$1\,096\,329$&$3174$&$393$&$111$&$27$&$13$&$2$&$3$&$2$&$0$&$1$\\
\bottomrule
\end{tabular}
\end{table}

At the opposite extreme there are the two familiar nonsingular cases $X_6\subset\P^5$ and $X_{10}\subset\P(5,1,1,1,1,1)$, and five cases having only $\frac{1}4(1,1,1,1)$ singularities:
\begin{gather*}
X_9\subset\P(4,1,1,1,1,1),
\quad
X_{12}\subset\P(4,4,1,1,1,1),
\quad
X_{15}\subset\P(5,4,3,1,1,1),\\
X_{16}\subset\P(8,4,1,1,1,1)
\quad\hbox{and}\quad
X_{24}\subset\P(12,8,1,1,1,1).
\end{gather*}
In total $32\,715$ of these Calabi--Yau fourfolds are singular with only isolated singularities.

\subsection{Fano fourfold hypersurfaces}\label{s!fano4}
A \emph{Fano fourfold} is a normal projective four-dimensional variety $X$ with $-K_X$ ample and with $X$ possessing at worst $\Q$-factorial terminal singularities. We embed such an $X$ in weighted projective space via its anticanonical ring $X=\Proj R(X)$, where $R(X) = \oplus_{m\in\N} H^0\mleft(X,-mK_X\mright)$.

\begin{thm}[Fano fourfold hypersurfaces]\label{thm!fano4}
If $X_d\subset\P(a_1,\dots,a_6)$ is a well-formed, quasismooth hypersurface with $\omega_X=\cO_X(-1)$ then, possibly after reordering the weights $a_i$, exactly one of the following two cases holds:
\begin{enumerate}
\item\label{fano4i}
$X_d$ is contained in one of $1597$ infinite series of the form
\[
X_{2k\sum b_i}\subset\P\mleft(-1 + k\sum_{i=1}^4 b_i,kb_1,kb_2,kb_3,kb_4,2\mright),
\qquad\text{ for all odd }k=1, 3, 5,\dots;
\]
\item\label{fano4ii}
$X_d$ is equal to one of $1\,233\,322$ sporadic cases.
\end{enumerate}
Of all $X$ in~\eqref{fano4i}--\eqref{fano4ii} there are exactly $11\,618$ cases that have terminal $\Q$-factorial singularities, and so are Fano fourfolds. The same number have canonical singularities.
\end{thm}

The complete list of sporadic cases and the infinite series is available online at~\cite{grdb}; the infinite series are summarised neatly by a correspondence due to Johnson and Koll\'ar in Corollary~\ref{cor!jkconj} below. We now describe some coarse features of the classification.

Four cases are nonsingular, namely
\[
X_5\subset\P^5,\quad
X_6\subset\P(2,1,1,1,1,1),\quad
X_8\subset\P(4,1,1,1,1,1),\quad\text{ and }\quad
X_{10}\subset\P(5,2,1,1,1,1),
\]
which match the well-known list of smooth canonical surface hypersurfaces; see~\eqref{eq!c2f}.

A typical Fano fourfold hypersurfaces contain both isolated and one-dimensional orbifold singularities. Besides the four nonsingular cases, $486$ have only isolated singularities, and $20$ have only one-dimensional singular loci. For example, $X_{1743}\subset\P(851, 581, 249, 41, 21, 1)$ has five isolated terminal quotient singularities, whilst $X_{20}\subset\P(10, 3, 3, 2, 2, 1)$ has two curves of transverse $\frac{1}{2}(1,1,1)$ and $\frac{1}{3}(1,1,2)$ points respectively.

The Fano fourfold hypersurface with the largest degree is the quintic $X_5\subset\P^5$. The one with the smallest degree, $K_X^4=1/498\,240\,036$,  is
\[
X_{3486}\subset\P(1743, 1162, 498, 42, 41, 1).
\]
This Fano fourfold is a double cover: $X_{3486}\stackrel{2:1}{\longrightarrow}\P(1162,498,42,41,1)$. In total $3511$ of the $11\,617$ Fano hypersurfaces arise naturally as double covers of weighted projective four-space.

\subsubsection*{Elephants}
Every anticanonical Fano threefold hypersurface $X$ has $h^0\mleft(X,-K_X\mright)\neq0$ and, for a general member, an effective anticanonical divisor $S\subset X$ can be chosen to be a K3 surface (one must allow Kleinian quotient singularities), a so-called \emph{general elephant}. The general elephant is central to the birational geometry of Fano threefolds: the $95$ Fano threefold hypersurfaces correspond directly to the $95$ K3 hypersurfaces and, for example, a great deal of calculation of their birational rigidity in Corti--Pukhilikov--Reid~\cite{CPR} takes place on the elephant.

For Fano fourfolds, the analogous question is whether or not there is a member $V\in\abs{-K_X}$ that is a Calabi--Yau threefold (allowing canonical singularities). In fact such an elephant need not exist for coarse reasons: plenty of Fano fourfolds have empty anticanonical system. The numbers of Fano fourfolds with $h^0\mleft(X,-mK_X\mright)=0$ for $m< b$ and $h^0(X,-bK_X)\neq0$ is given in Table~\ref{tab:fano4_b}. In particular, $1036$ Fano fourfold hypersurfaces have $\abs{-K_X}$ empty; the two extreme cases are $X_{120} \subset \P(40,24,21,15,11,10)$ and $X_{112} \subset \P(28,24,21,16,13,11)$.
\begin{table}[ht]
\caption{The number of Fano fourfold hypersurfaces with smallest weight $b$.}\label{tab:fano4_b}
\centering
\begin{tabular}{rGwgwgwgwgwg}
\toprule
$b$&$1$&$2$&$3$&$4$&$5$&$6$&$7$&$8$&$9$&$10$&$11$\\
\#&$10\,581$&$645$&$244$&$80$&$42$&$9$&$10$&$3$&$1$&$1$&$1$\\
\bottomrule
\end{tabular}
\end{table}

For the remaining $10\,581$ cases we compare the list of Fano fourfold hypersurfaces with the Kreuzer--Skarke classification of Calabi--Yau threefold hypersurfaces. The number of Fano fourfold hypersurfaces grouped according to $P_1=h^0\mleft(X,-K_X\mright)$, and Calabi--Yau threefold hypersurfaces $Y$ with $h^0\mleft(Y,\cO_Y(1)\mright)=P_1-1$, is given in Table~\ref{tab:Fano4_CY3}.
\begin{table}[ht]
\caption{The number of Fano fourfold hypersurfaces $X$ with $P_1=h^0\mleft(X,-K_X\mright)$ and Calabi--Yau threefold hypersurfaces $Y$ with $h^0\mleft(Y,\cO_Y(1)\mright)=P_1-1$.}
\label{tab:Fano4_CY3}
\centering
\begin{tabular}{rGWGwgwgW}
\toprule
$P_1$&$0$&$1$&$2$&$3$&$4$&$5$&$6$&Total\\
\cmidrule{1-9}
\#Fano fourfolds&$1036$&$8697$&$1772$&$98$&$11$&$3$&$1$&$11\,618$\\
\#Fanos with CY3&$0$&$7953$&$1772$&$98$&$11$&$3$&$1$&$9838$\\
\#Fanos with quasismooth CY3&$0$&$6014$&$1429$&$97$&$11$&$3$&$1$&$7555$\\
\#Kreuzer--Skarke CY3&n/a&$168\,813$&$15\,016$&$179$&$14$&$3$&$1$&$184\,026$\\
\bottomrule
\end{tabular}
\end{table}
The Fano/Calabi--Yau correspondence is most striking for those $8697$ Fano fourfolds with $h^0\mleft(X,-K_X\mright)=1$, since these have a unique effective anticanonical divisor. In $6014$ cases there is a corresponding quasismooth Calabi--Yau, and if $X$ is chosen generally in its family, then its anticanonical divisor is a quasismooth Calabi--Yau threefold. Some of the $8697$ cannot have a quasismooth elephant because they possess a singularity not polarised with a~one. For example, $X_{16}\subset\P(5,4,3,2,2,1)$ has a singularity $\frac{1}{5}(2,2,3,4)$; put another way, the degree~one variable is the only possible tangent form at the index~five point. The elephant is still a Calabi--Yau threefold (for general such $X$), just not quasismooth.

There are also cases where the anticanonical divisor is never Calabi--Yau. For example, $X_{23}\subset\P(11,4,3,3,2,1)$ has anticanonical divisor $S_{23}\subset\P(11,4,3,3,2)$. This has a hyperquotient singularity of type $\frac{1}{11}(4,3,3,2;1)$, in the notation of~\cite[(4.2)]{YPG}. But any singularity of this type is not canonical: the given weighted blowup of the ambient space, of discrepancy only $1/11$, extracts a divisor of discrepancy $\le-1$ from the threefold.

The Fano $X_{43}\subset\P(21,14,6,1,1,1)$ is slightly different. Both this and any effective anticanonical divisor must contain the plane $\P(21,14,6)$. For the fourfold, this is no problem; for its threefold elephant, this forces up the divisor class group, the price of which is a non-$\Q$-factorial point, so certainly not quasismooth. This Fano fourfold is the `$98$th' with $P_1=3$
in the table above, which does not match one of the $97$ quasismooth Calabi--Yaus.

\subsubsection*{Tigers}
Tigers were introduced by Keel--McKernan~\cite{KM99}. We follow Johnson--Koll\'ar~\cite[Definition 3.2]{JK2}: on normal variety $X$, a \emph{tiger} is an effective $\Q$-divisor $D$, numerically equivalent to $-K_X$, for which the pair $(X,D)$ is not Kawamata log terminal.

A sufficient condition under which $X_d$ does not admit a tiger is given in~\cite[Proposition~3.3]{JK2}: if $X_d\subset\P(a_1,\dots,a_s)$ with $a_1\ge a_2\ge\cdots\ge a_s$ then $X$ does not have a tiger if $d\le a_{s-1}a_s$. This condition never holds for hypersurfaces in the infinite series, but in the sporadic hypersurfaces in Theorem~\ref{thm!fano4}\eqref{fano4ii} there are $443\,485$ cases which  satisfy this condition, and so for which no member has a tiger. Only one of these has terminal singularities, namely $X_{112}\subset\P(28,24,21,16,13,11)$.

\subsubsection*{K\"ahler--Einstein metrics}
Johnson--Koll\'ar~\cite[Proposition 3.3]{JK2} give conditions under which $X_d$ admits a K\"ahler--Einstein metric: if $X_d\subset\P(a_1,\dots,a_s)$ with $a_1\ge a_2\ge \cdots \ge a_s$ then $X$ admits a K\"ahler--Einstein metric if $d<\frac{s-1}{s-2} a_{s-1}a_s$.

Of the sporadic hypersurfaces in Theorem~\ref{thm!fano4}\eqref{fano4ii} there are $490\,083$ cases for which every quasismooth member admits a K\"ahler--Einstein metric. Eight of these have terminal singularities:
\begin{align*}
X_{77}&\subset\P(20, 17, 14, 11, 9, 7),
&X_{80}&\subset\P(20, 16, 15, 13, 10, 7),\\
X_{80}&\subset\P(20, 16, 15, 13, 9, 8),
&X_{90}&\subset\P(25, 18, 15, 14, 13, 6),\\
X_{90}&\subset\P(30, 18, 13, 12, 11, 7),
&X_{91}&\subset\P(25, 20, 16, 13, 11, 7),\\
X_{112}&\subset\P(28, 24, 21, 16, 13, 11),
&X_{120}&\subset\P(40, 24, 21, 15, 11, 10).
\end{align*}

\subsection{The Johnson--Koll\'ar conjecture}\label{s!jkconj}
Johnson and Koll\'ar~\cite{JK2} relate three classes, which we describe below. In each case we define $c:=b_1+\cdots+b_{s-2}$.
\begin{enumerate}
\item\label{jkT}
Well-formed weighted projective $(s-3)$-spaces $\P=\P(b_1,\dots,b_{s-2})$ which admit a quasismooth member $T\in\abs{-2K_\P}$; in other words, those $\P$ for which there is a (not necessarily well-formed) quasismooth $(s-4)$-fold hypersurface
\[
T_{2c} \subset \P(b_1,\dots,b_{s-2}).
\]
\item\label{jkcy}
Well-formed quasismooth Calabi--Yau $(s-3)$-fold hypersurfaces
\[
S_{2c} \subset \P(c,b_1,\dots,b_{s-2}).
\]
\item\label{jkseries}
One-parameter series of well-formed quasismooth hypersurfaces
\[
X^{(k)}_{d^{(k)}}\subset\P(a_1^{(k)}, \dots, a_s^{(k)})\qquad\text{ for odd }k\in\N,
\]
whose weights are determined by the series
\[
(a_1^{(k)},a_2^{(k)},\dots,a_s^{(k)},d^{(k)} ) =
\left(c-1,b_1,\dots,b_{s-2},2,2c\right) + k\left(c,b_1,\dots,b_{s-2},0,2c\right).
\]
\end{enumerate}
It is clear that if $\P(b_1,\dots,b_{s-2})$ is well-formed in~\eqref{jkT} then so are the corresponding $S$ and~$X$. That quasismoothness passes from $T$ to $X$ is shown in~\cite[Lemma~2.4]{JK2}. Indeed, given quasismooth $T\colon(f(z_1,\dots,z_{s-2})=0)\subset\P$, the hypersurface $S\colon(y_1^2=f(y_2,\dots,y_{s-1}))\subset\P(\sum b_i,b_1,\dots,b_{s-2})$ is quasismooth; conversely, given a quasismooth $S$, completing the square reveals a polynomial $f$ that defines a quasismooth $T$. Continuing with such an $S$,  the hypersurface
\[
X_{2k\sum b_i}\colon\left(x_1^2x_s=f(x_2,\dots,x_{s-1})+x_s^{k\sum b_i}\right)\subset\P\mleft(-1+k\sum b_i,kb_1,\dots,kb_{s-2},2\mright)
\]
is quasismooth, and so the general hypersurface $X$ is too. However, there is no clear reason why an arbitrary quasismooth series of~$X$ of that degree should arise from a quasismooth~$S$.

Nevertheless, Johnson and Koll\'ar conjecture that if $X^{(k)}_{d^{(k)}}\subset\P(a_1^{(k)},\dots,a_s^{(k)})$ is a one-parameter series containing infinitely many well-formed, quasismooth hypersurfaces whose weights are determined by a series
\[
(a_1^{(k)},a_2^{(k)},\dots,a_s^{(k)},d^{(k)})=(v_1,\dots,v_{s+1})+k(w_1,\dots,w_{s+1}),
\]
for fixed vectors $v$ and $w$, then it is of type~\eqref{jkseries} above, and furthermore it arises from a quasismooth $T$ of type~\eqref{jkT} lying in a well-formed weighted projective space $\P(b_1,\dots,b_{s-2})$.

Well-formedness of $S$ in~\eqref{jkcy} is not, by itself, sufficient to imply that the varieties in~\eqref{jkT} and~\eqref{jkseries} are well-formed. However, if $S$ is well-formed then $\gcd{b_1,\dots,b_{i-1},b_{i+1},\dots,b_{s-2}}=1$ or~$2$ for any~$i$, hence the only way for $\P(b_1,\dots,b_{s-2})$ to fail to be well-formed is by having an index~two stabiliser in codimension~one. Exactly that failure is catastrophic for $X$, so we can rephrase the conjecture informally as: infinite series in dimension~$n$ with $\omega_X=\cO_X(-1)$ correspond to hyperelliptic Calabi--Yau $(n-1)$-folds that do not have a locus of transverse $\frac{1}{2}(1,1)$ quotient singularities in codimension~two that is fixed pointwise by the involution.

In dimension~three there are $48$ well-formed, quasismooth K3 surfaces of the form $S_{2\sum b_i}\subset\P(\sum b_i,b_1,b_2,b_3)$. Of these, $25$ of the spaces $\P(b_1,b_2,b_3)$ are well-formed, and these correspond to the $25$ one-parameter series in Table~\ref{tab!3d} with $d=3$, $k=-1$.

\begin{cor}\label{cor!jkconj}
The Johnson--Koll\'ar conjecture holds when $\dim X=4$.
\end{cor}

To check the corollary, it suffices to compare the list of $1597$ series in Theorem~\ref{thm!fano4} with the list of $7555$ quasismooth Calabi--Yau threefold hypersurfaces. Within the latter, $2390$ cases are double covers $S_{2\sum b_i}\subset\P(\sum b_i,b_1,b_2,b_3,b_4)$. Of these, $1597$ of the spaces $\P(b_1,b_2,b_3,b_4)$ are well-formed. These correspond to the $1597$ one-parameter series with $d=4$, $k=-1$, exactly as predicted by the conjecture.

\subsection{Orbifold rational curves and ADE singularities}\label{sec!ADE}
To illustrate the basic method of classification, we will consider quasismooth projective curves
\[
C_d\colon(F_d(x,y,z)=0)\subset\P(a,b,c)\quad\text{ with }\quad d+1 = a+b+c.
\]
Here we write $F=F_d$ for a form of degree~$d$. The numerical condition implies that $g_C=0$. To carry out the analysis we enforce the order $a\ge b\ge c$ on the ambient weights.

Quasismoothness at $P_x=(1:0:0)$ implies that at least one of $x^m$, $x^{m-1}y$, or $x^{m-1}z$ appears (with non-zero coefficient) in $F$, for some $m\ge 2$. In each case the monomial determines the degree $d$, and then the numerical condition implies that $m<3$: indeed $a\ge b$ so $2a+c\ge d+1$, but $d\ge (m-1)a+c$. We consider the three cases $x^2$, $xy$, and $xz\in F$ separately.

If $x^2\in F$ then $d=2a$ and the condition reads $a+1=b+c$. Quasismoothness at $P_y$ implies that at least one of $y^{m-1}x$, $y^m$, or $y^{m-1}z$ appears in $F$. The relation gives that $2b\ge a+1$, so that $3b+c\ge d+2$ and $m\le 3$. Thus the possible monomials at $P_y$ are $y^3$, $y^2z$, $y^2$, and $yz$. (Other cases are dismissed at once: $xy^2$ implies $d=a+2b\ge a+b+c=d+1$, whilst $xy$ will be considered separately at $P_x$.)

Again we consider the cases separately. If $x^2+y^3\in F$ then we can assemble the numerical conditions together in a standard auxiliary matrix:
\[
N =
\left(\begin{array}{cccc|c}
1 & 1 & 1 & -1 & 1 \\
2 & 0 & 0 & -1 & 0 \\
0 & 3 & 0 & -1 & 0
\end{array}
\right),
\]
where any positive integral element of the kernel of $N$ of the form $(a,b,c,d,1)$ provides weights that satisfy the numerical conditions so far. We compute the kernel using the (integral) echelon form of $N$, which in this case is
\[
\left(\begin{array}{cccc|c}
1 & 0 & 3 & -1 & 3 \\
0 & 1 & 4 & -1 & 4 \\
0 & 0 & 6 & -1 & 6
\end{array}
\right).
\]
This provides at once the solutions so far, namely a one-dimensional series
\[
C_{6n}\colon(x^2+y^3+\cdots=0)\subset\P(3n,2n,n+1),\qquad\text{ where }n\ge 1.
\]
For small $n$ we understand these easily:
\begin{center}
\vspace{0.5em}
\setlength{\extrarowheight}{0.2em}
\begin{tabular}{cccc}
\toprule
$n$&Type&Typical equation&Numerics\\
\cmidrule{1-4}
\oddrow $1$&$\D{4}$&$x^2 + y^3 + z^3$&$C_6\subset\P(3,2,2)$\\
\evnrow $2$&$\E{6}$&$x^2 + y^3 + z^4$&$C_{12}\subset\P(6,4,3)$\\
\oddrow $3$&$\E{7}$&$x^2 + y^3 + yz^3$&$C_{18}\subset\P(9,6,4)$\\
\evnrow $5$&$\E{8}$&$x^2 + y^3 + z^5$&$C_{30}\subset\P(15,10,6)$\\
\bottomrule
\end{tabular}
\vspace{0.5em}
\end{center}
The case $n=4$ is excluded since any curve $C_{24}\subset\P(12,8,5)$ fails to be quasismooth at $P_z$. In fact, this requirement at $P_z$ implies that at least one of $xz^{m-1}$, $yz^{m-1}$, and $z^m$ appears in $F$, which in turn implies that
\[
m=\frac{ne}{n+1},\qquad\text{ where }e\le6,
\]
so $m<6$ and this branch of the tree search is complete.

A similar exercise in the case when $x^2+y^2z\in F$ leads to the echelon form
\[
\left(\begin{array}{cccc|c}
1 & 1 & 0 & -1 & -1 \\
0 & 2 & 0 & -1 & -2 \\
0 & 0 & 1 & 0 & 2
\end{array}
\right)
\]
which gives
\[
C_{2n-2}\colon (x^2+y^2z+z^{n-1}=0)\subset\P(n-1,n-2,2),\qquad\text{ where }n\ge 4.
\]
The equations in this case are the familiar Type~$\D{n}$ equations. All other branches of the tree provide $F$ with quadratic terms of rank at least two, and lead to Type~$\A{n}$ equations. Hence we recover precisely the classification of ADE singularities as the affine cones on the resulting curves $C$.

This simple exercise illustrates most of the ideas. In particular, even though we may have infinitely many solutions, the collections of all monomials that can give quasismoothness is finite and can be bounded at each step of the tree search. There is, however, an additional consideration: the hypersurfaces listed here are not well-formed. In the analysis below we enforce that additional condition too.

\section{The quasismooth algorithm}\label{sec!alg}
We search for collections of integers
\[
(a_1,\dots,a_s,d)\quad\text{ with }\quad a_1\ge\cdots\ge a_s\ge 1
\]
for which, for a general form $F_d$ of degree~$d$, the hypersurface
\[
X_d\colon(F_d(x_1,\dots,x_s)=0) \subset \P^{s-1}(a_1,\dots,a_s)
\]
is well-formed and quasismooth. We use this notation throughout, including the ordering on the $a_i$. We use $P_i = (0:\cdots :1:0:\cdots:0)$, where the $1$ is in the $i$th position, to denote the $i$th coordinate point of $\P^{s-1}(a_1,\dots,a_s)$. (Note that the variables start with $x_1$, not the usual $x_0$, and the ordering convention on the $a_i$ is the opposite of that in~\cite{JK2}.)

We are mainly interested in those varieties whose singularities lie in some restricted class; see~\S\ref{sec!results} for details. First, however, we make a complete list of the well-formed quasismooth hypersurfaces, irrespective of singularities. The conditions on the singularities are imposed afterwards. This is the same approach as taken in~\cite{JK,JK2}.

We learned the basic algorithm below from Miles Reid, who used it in~\cite{C3f} to compute the `famous $95$'. The same approach is used by~\cite{JK,JK2}: quasismoothness at $0$-strata imposes a range of possible linear conditions on integral vectors $(a_1,\dots,a_s,d)$ that one must organise and solve.

\subsection{Inequalities with semipositive and seminegative summation}
Let $a_1\ge a_2\ge\cdots\ge a_s\ge 1$ be integers, and let $p_1,\dots,p_s$ be any integers (here negative integers are allowed). We use the notation $\sump p_\ell$ to denote the following number:
\[
\sump p_\ell=\sum_{\ell=1}^s p_\ell',
\]
where either $p_\ell'=p_\ell\ge0$ or $0\ge p_\ell'\ge p_\ell$ is chosen minimally so that each sum $p_\ell'+p_{\ell+1}'+\cdots+p_s'\ge 0$, working inductively down through $\ell=s,s-1,\dots,i$. Of course, if all $p_\ell\ge 0$ then $\sump p_\ell =\sum p_\ell$.

Equivalently, define the sequence $\sigma_0:=\min{p_s,0}$ and $\sigma_{i+1}:=\min{\sigma_i+p_{s-i-1},0}$. Then $\sump p_\ell=\sigma_s$. We also define $\sumn p_\ell$. Set $\tau_0:=\max{p_s, 0}$ and $\tau_{i+1}:=\max{\tau_i + p_{s-i-1}}$. Then $\sumn p_\ell=\tau_s$. The following lemma is elementary.

\begin{lem}\label{lem!sump}
In the notation above,
\[
a_1 \sump p_i \ge a_1 p_1 + \cdots + a_s p_s
\]
and
\[
a_1 \sumn p_i \le a_1 p_1 + \cdots + a_s p_s.
\]
\end{lem}

\subsection{Quasismoothness at $0$-strata}
At each coordinate point $P_i$, quasismoothness requires a monomial $x_i^{m_i-1}x_{j_i}\in F$ for some $m_i\in\N$ and some $1\le j_i \le n$ (the case $j_i=i$ occurs when $P_i\notin X$). These conditions do not necessarily restrict the search to a finite set of solutions -- as in~\cite{JK} we expect some infinite series -- but they do allow for a terminating search provided that we can handle certain infinite series. By recursively imposing these conditions at every point, we obtain a branching search-tree. The monomials appearing in $F$ at each step are recorded in a matrix of exponents; this same matrix allows us to recognise when a branch of the tree has been exhausted.

We compare the degree of a Laurent monomial $x_1^{n_1}\cdots x_s^{n_s}$, where each $n_i\in\Z$, with the degree of $F$ by recording the integers appearing in the expression
\[
\deg\frac{x_1^{n_1}\cdots x_s^{n_s}}{F^m}=e
\]
as a row of integers
\[
(n_1,\dots,n_s,-m,e).
\]
We recursively build a matrix using these rows of integers, as explained below. We refer to this matrix as the \emph{tangent monomial matrix}.

We always start in the same way by comparing the monomial $x_1\cdots x_s$, which has degree $a_1+\cdots+a_s$, with $F$ itself; in this case the relative degree is minus the canonical degree $k$ and we record the first row of the tangent monomial matrix, which has $s+2$ columns, as
\[
(1,\dots,1,-1,-k).
\]

Any polynomial $F$ whose variety $X = (F=0)$ is quasismooth at $P_1$ must include a monomial of the form $x_1^{m_1}x_{j_1}$ for some non-negative integer $m_1$ and for some variable $x_{j_1}$; $j_1=1$ is allowed, and then $P_1\notin X$, which is fine. The crucial observation is that there are only finitely many pairs $(m_1,j_1)$ that can arise in this way. The number of pairs is bounded by the following lemma.

\begin{lem}
Assume that $s\ge 4$.
\begin{enumerate}
\item\label{item:finiteness1}
If $k<0$ then $m_1<s$, and so the possible pairs $(m_1,j_1)$ are
\[
\left\{(m,j)\mid m\in\{2,3,\dots,s-1\},j\in\{1,2,\dots,s\}\right\}.
\]
If $2-s<k<0$ then the only case with $m_1=2$ that could arise is $(2,1)$.
\item\label{item:finiteness2}
If $k\ge 0$ then $m_1\le s+k$, and so the possible pairs $(m_1,j_1)$ are
\[
\left\{(m,j)\mid m\in\{2,3,\dots,s+k\},j\in\{1,2,\dots,s\}\right\}.
\]
\end{enumerate}
\end{lem}

\begin{proof}
The proof uses the numerics of adjunction: $k=d-(a_1+\cdots+a_s)$. If $X$ is quasismooth at $P_1$ we need a \emph{tangent monomial} $x_1^{m_1-1}x_j\in F$, for some $m_1\ge 2$ and $j\in\{1,2,\dots,s\}$. Suppose that $k<0$. If $m_1\ge s$ then
\[
d=\deg(x_1^{{m_1}-1}x_j)=(m_1-1)a_1+a_j\ge\sum a_i=d+(- k)>d,
\]
which is a contradiction. Thus a tangent monomial in $F$ of degree $m_1$ requires $m_1<s$.

Suppose now that $k\ge 0$. If $x_1^{m_1-1}x_j\in F$ with $m_1\ge s+k+1$ then
\[
d=(m_1-1)a_1+a_j\ge\sum a_i+(m_1-s)\ge(d-k)+(k+1)=d+1,
\]
which again is a contradiction. Hence $m_1\le s+k$ as required.

The additional restriction in~\eqref{item:finiteness1} when $-k<s-2$ follows directly from adjunction, since otherwise $d-k=a_1+a_j-k<\sum a_i$.
\end{proof}

For each pair $(m_1,j_1)$ we use the row of exponents and degree to extend the matrix, so from the first monomial we obtain the new tangent monomial matrix
\[
\begin{pmatrix}
1 & 1 & \cdots & 1 & 1 & 1 & \cdots & 1 & -1 & -k \\
m_1-1 & 0 & \cdots & 0 & 1 & 0 & \cdots & 0 & -1 & 0
\end{pmatrix},
\]
where the $1$ in the second row is in the $j_1$th position (we have illustrated the case when $j_1\neq 1$, but this is also a possibility). The integral echelon reduction of this matrix is
\[
\begin{pmatrix}
1 & 1 & \cdots & 1 & 1 & 1 & \cdots & 1 & -1 & -k \\
0 & m_1-1 & \cdots & m_1-1 & m_1-2 & m_1-1 & \cdots & m_1-1 & 2-m_1 & k(1-m_1)
\end{pmatrix},
\]
with a further simplification of the top row possible if $m_1=2$.

Set $c=-(2-m_1)\ge0$ and $b=k(1-m_1)$ -- the final two entries of the last row, with the indicated change of sign -- and assemble all possible pairs $(m_2,j_2)$ for which a monomial $x_2^{m_2-1}x_{j_2}$ could exist to verify the quasismoothness of $X$ at $P_2$. The possible pairs $(m_2,j_2)$ are determined by considering the case when $i=2$ in the following lemma.

\begin{lem}\label{lem!max}
Suppose that when considering the $i$th $0$-stratum $P_i$, the last row of the tangent monomial matrix in echelon form is
\[
(\stackrel{i-1}{\overbrace{0,\dots,0}},p_i,p_{i+1},\dots,p_s,-c,b),
\]
where the first $i-1$ entries are zero, and $p_i>0$.
\begin{enumerate}
\item\label{item:lem!max_1}
If $c>0$ then
\[
m_i\le\begin{cases}
\left\lceil(\sump p_\ell)/c\right\rceil,&\text{ if }b\ge 0;\\
\left\lceil(\sump p_\ell-b)/c\right\rceil,&\text{ otherwise.}
\end{cases}
\]
\item\label{item:lem!max_2}
If $c<0$ and all $p_\ell\ge0$ then
\[
m_i\le\begin{cases}
\left\lceil(\sumn p_\ell-b)/c\right\rceil,&\text{ if }b\ge 0;\\
\left\lceil(\sumn p_\ell)/c\right\rceil,&\text{ otherwise.}
\end{cases}
\]
\item\label{item:lem!max_3}
If $c=0$ and all $p_\ell\ge 0$ then there are no solutions if $\sum_{\ell=i}^s p_\ell>b$.
\end{enumerate}
\end{lem}

\begin{proof}
Recall the notation: $a_\ell=\deg x_\ell$ are the unknown positive integers that we are attempting to solve for, ordered in decreasing order. We seek possible $(m,j)$ such that $x_i^{m-1}x_j\in F$; in particular, such monomials have the same degree as $F$. Given such a choice of $(m,j)$, the current last row of the matrix
\[
\left(0,\dots 0,p_i,p_{i+1},\dots,p_s,-c,b\right)
\]
implies that
\begin{align}
\nonumber
\deg \frac{x_i^{p_i}\cdots x_s^{p_s}}{\left(x_i^{m-1}x_j\right)^c}&=b\\
\label{eq!degbound}
\text{ hence }\quad a_ip_i+\cdots+a_sp_s-ca_i(m-1)-ca_j&=b.
\end{align}
Suppose first that $c>0$. By Lemma~\ref{lem!sump},
\[
a_i\sump p_\ell-ca_i(m-1)-ca_j\ge b.
\]
Hence $a_i\left(\sump p_\ell-c(m-1)\right)>b$. When $b\ge 0$ we obtain $(\sump p_\ell)/c>m-1$, which gives the first bound in~\eqref{item:lem!max_1}. When $b<0$, dividing by $a_i$ gives
\[
\sump p_\ell-c(m-1)>b/a_i\ge b,
\]
which gives the second bound in~\eqref{item:lem!max_1}.

Suppose now that $c<0$. The inequality~\eqref{eq!degbound} together with Lemma~\ref{lem!sump} implies that
\[
a_i\sumn p_\ell+\abs{c}a_i(m-1)+\abs{c}a_j\le b.
\]
When $b\ge 0$ we obtain
\[
\sumn p_\ell + \abs{c}(m-1) < b / a_i \le b,
\]
giving the first bound in~\eqref{item:lem!max_2}. When $b<0$, using $b/a_i\le 0$ gives the second bound.

When $c=0$ the same analysis does not give a relation between the degree of $F$ and that of the monomial $x_i^{p_i}\cdots x_s^{p_s}$, so we do not get a bound on $m$. Nevertheless, if $\sum p_\ell>b$ then that monomial cannot have degree~$b$ for any (positive integral) choice of weights; in this case we conclude that there are no solutions, giving~\eqref{item:lem!max_3}.
\end{proof}

\begin{cor}\label{cor!max}
Suppose that $c\neq0$ and $i>1$. Let $\mmax$ be the upper bound for the $m_i$ in Lemma~\ref{lem!max}, determined according to the signs of $b$ and $c$. Then the possible pairs $(m_i,j_i)$ are
\[
\left\{(m,j)\mid m\in\{2,3,\dots,\mmax\}, j\in\{1,2,\dots,s\}\right\}.
\]
\end{cor}

\noindent
If $c\neq0$ then Corollary~\ref{cor!max} provides the possible choices for the next row of the tangent monomial matrix. We run through each of these in turn, repeating this step until either $c=0$ or $i=s-1$. When either case occurs we move to the next step, described below, which is to use this system of equations encoded in the matrix to find all possible systems of weights.

\subsection{Solving for possible weights}
At this stage the tangent monomial matrix $N$ is of size $r\times s+2$, for $2\le r \le s$. Note that it can happen that $c=0$ and we stop growing $N$ before it has $s$ rows. Treating this as an auxiliary matrix, we solve the $r$ inhomogeneous equations in $s+1$ unknowns $(a_1,\dots,a_s,d)$ -- the inequalities $a_1\ge\cdots\ge a_s\ge 1$ remain in force to avoid repeating solutions and we solve for integral points of the resulting polyhedron. There is no reason why these solutions should represent well-formed and quasismooth hypersurfaces, so we perform some additional checks to eliminate infinite polyhedrons that could only contribute finitely many solutions.

\subsubsection*{The hyperbola trick}
We repeatedly use the following ``hyperbola trick''. Let $a,b,c,d$ be fixed integers, and for simplicity suppose that $d\neq0$. Consider the expression
\[
N=\frac{a+\lambda b}{c+\lambda d}.
\]
What is the largest value of $\lambda\in\Z$ for which $N$ is an integer? Sketching the graph of $N$ as a function of $\lambda$ answers this at once. The geometry is controlled by the determinant
\[
\Delta=\det\!\begin{pmatrix}a&b\\c&d\end{pmatrix}.
\]
If $\Delta=0$ then $N=b/d$ is constant, and $\lambda$ is bounded (in fact has no solutions at all) if and only if $b/d\notin\Z$. If $\Delta>0$ then $dN/d\lambda<0$ and the graph approaches the asymptote $b/d$ from above. In this case $\lambda$ is bounded above by setting $N$ to be the smallest integer greater than $b/d$, that is, $N=\lfloor b/d+1\rfloor$, and solving for $\lambda$. If $\Delta<0$, the graph approaches $b/d$ from below, and $\lambda$ is bounded by setting $N$ to be the largest integer less than $b/d$, that is, $N=\lceil b/d-1\rceil$.

\subsection{One-parameter series of solutions}
Consider the case when the solution polyhedron is one-dimensional, with integer points $\left\{u+\lambda k\mid \lambda\in\N\right\}$, for some $u=(u_1,\dots,u_{s+1})$, $k=(k_1,\dots,k_{s+1})\in\Z^{s+1}$, where $u$ lies in the strictly positive quadrant. Consider a general solution $(u_1+\lambda k_1,\dots,u_{s+1}+\lambda k_{s+1})$. We describe the tests we subject this series to in our implementation; there is some overlap.

\subsubsection*{Quasismooth Test I: the final coordinate point}
Suppose $k_s\ne 0$, and consider the point $P_s$. If $x_i$ is to be a tangent monomial at $P_s$ (for any $i$, including the possibility that $i=s$ when $P_s$ does not lie on the hypersurface) then $x_s^N x_i\in F$, for some $N\in\N$. Computing degrees and rearranging gives
\begin{equation}\label{eq!N}
N=\frac{(u_{s+1}-u_i)+\lambda(k_{s+1}-k_i)}{u_s+\lambda k_s}.
\end{equation}
We now apply the hyperbola trick described above. In this case $N$ tends to $(k_{s+1}-k_i)/k_s$ as $\lambda\rightarrow\infty$, either from above or from below depending on the sign of the determinant
\[
\det\!\begin{pmatrix}
u_{s+1}-u_i&k_{s+1}-k_i\\
u_s&k_s
\end{pmatrix}.
\]
In either case this gives a formula for the largest value of $\lambda$ for which $N$ in equation~\eqref{eq!N} is integral. If the determinant is zero then the hyperbola consists of two lines, one with $\lambda=-u_s/k_s<0$ (which gives no solutions) and one with $N=(k_{s+1}-k_i)/k_s$, which gives solutions if and only if $N$ is integral; in this case the series requires further analysis.

When the determinant is not zero, the maximum of these values for $1\le i\le s$ gives an upper bound of $\lambda$, and so this is not an infinite series of solutions after all: we compute the finitely many cases as sporadic solutions.

For example, consider the case $u=(20, 9, 4, 4, 4, 40)$ and $k=(15, 7, 3, 3, 2, 30)$.
With this input, the determinant above is never zero, and this method detemines
the maximum $\lambda=16$. When interpreted as a weighted hypersurface, $u+16k$
is indeed quasismooth; however, in this case it does not represent a well-formed hypersurface, so will be
excluded at at later stage. The cases $\lambda=1$, 3 and~13 all give rise to well-formed,
quasismooth hypersurfaces:
\[
X_{70}\subset\P(35, 16, 7, 7, 6),\quad
X_{130}\subset\P(65, 30, 13, 13, 10),\quad
X_{430}\subset\P(215, 100, 43, 43, 30).
\]

\subsubsection*{Almost identical $u$, $k$}
If $u$ and $k$ have $s-1$ of the first $s$ entries in common, then no case beyond the first point $u$
is well-formed, so we may reject the rest of series and treat $u$ as a sporadic case. If they have only $s-2$ entries in common, and $v=u_{s+1}-k_{s+1}\ne0$, then the ambient space has a codimension~two stratum with nontrivial stabiliser, so the equation must prevent that lying inside~$X$. The only multiples of $\lambda$ that permit this are zero and the non-unit divisors of $v$.

\subsubsection*{Complementary $u$, $k$ modulo $2$}
A parity check on the sum of entries of $u$ and $u+k$ also rules out the whole series in cases where the stabiliser $\Z/2$ fixes a large coordinate subspace.

\subsubsection*{Quasismooth Test II: one-dimensional strata}
If exactly one $k_\ell=0$ then we check for one-strata with equal $u_i$ and equal $k_i$ where we can apply the hyperbola trick to bound $\lambda$.

Consider any one-stratum $\left<x_{i_1},x_{i_2}\right>$ with $k_{i_1}=k_{i_2}$ and $u_{i_1}=u_{i_2}$. Suppose further that $k_{s+1}$ is divisible by $k_{i_1}$, but that $u_{s+1}$ is not divisible by $u_{i_1}$. This implies that the one-stratum is contained in every hypersurface of the series. If this series really does contain infinitely many quasismooth members then there must be two tangent variables that work (numerically, at least) for infinitely many $\lambda$.

Suppose $x_{i_1}^{N_1}x_{i_2}^{N_2}x_j\in F$. Calculating degrees and rearranging gives
\[
N_1+N_2=\frac{(u_{s+1}-u_j)+\lambda(k_{s+1}-k_j)}{u_{i_1}+\lambda k_{i_1}}.
\]
The hyperbola trick applies. If for every $j=1,\dots,s$, $j\neq i_1,i_2,\ell$, the associated determinant is non-zero then even requiring a single tangent variable along this one-stratum puts an upper bound on $\lambda$. (This one-stratum needs two tangent variables, which is why we do not impose the conditions on $x_\ell$.) We can calculate this bound and regard all cases below it as sporadic cases.

\subsubsection*{Easy codimension two failure of well-formedness}
Let $(u_1+\lambda k_1,\dots,u_{s+1}+\lambda k_{s+1})$ be a one-parameter series solution.
Denote a general polynomial that defines the corresponding hypersurface by $F_\lambda$.

\begin{lem}\label{lem!wfr}
Suppose that there exists a subset $I=\{i_1,\dots,i_r\}\subset\{1,\dots,s-1\}$ and $\alpha>0$ such that $u_i=\alpha k_i$ for all $i\in I$. Denote $d=u_{s+1}$ and $e=k_{s+1}$. If for some $\lambda>0$ there is a monomial $m=x_{i_1}^{p_1}\cdots x_{i_r}^{p_r}\in F_\lambda$ for which $p_1k_{i_1}+\cdots+p_rk_{i_r}=e$ then
\begin{enumerate}
\item\label{lemi}
$p_1u_{i_1}+\cdots+p_ru_{i_r}=d$,
\item\label{lemii}
$m\in F_\lambda$ for all $\lambda\ge 0$, and
\item\label{lemiii}
$d=\alpha e$.
\end{enumerate}
\end{lem}

\begin{proof}
Calculating the degree of the given monomial for the given $\lambda>0$ gives
\[
\sum_{j=1}^rp_j(u_{i_j}+\lambda k_{i_j})=d+\lambda e
\]
which, after rearrangement, gives
\begin{equation}\label{eq!lam}
\lambda\left(\sum_{i=1}^rp_jk_{i_j}-e\right)=d-\sum_{i=1}^rp_ju_{i_j}.
\end{equation}
Part~\eqref{lemi} follows immediately. Since this equation holds independently of~$\lambda$, we obtain~\eqref{lemii}. Finally, by substituting $u_{i_j}=\alpha k_{i_j}$ into~\eqref{lemi} we obtain~\eqref{lemiii}.
\end{proof}

The special case $r=s-2$ provides a well-formedness test, generalising the `almost identical' test above. Suppose that $u+\lambda k$ is a one-parameter series as in Lemma~\ref{lem!wfr} for which $d\neq\alpha e$. In this case, whenever $\lambda>0$, the corresponding stratum $\P_I$ has nontrivial stabiliser and so cannot be contained in a well-formed hypersurface~$X$. Therefore there must be a monomial $m\in F_\lambda$, as in Lemma~\ref{lem!wfr}\eqref{lemii}. To avoid a contradiction, we must have that $\sum p_ik_i-e\neq 0$. But rearranging~\eqref{eq!lam} for $\lambda$ gives
\[
\lambda=\frac{d-\sum p_iu_i}{\sum p_ik_i-e}
=\frac{(d-\alpha e)+\alpha e-\sum p_i\alpha k_i}{\sum p_ik_i-e}
=\frac{d-\alpha e}{\sum p_ik_i-e}-\alpha
\le\abs{d-\alpha e}-\alpha.
\]
This gives an upper bound on $\lambda$, so we may reject the one-parameter series and consider instead the finite number of solutions having these $\lambda$ as sporadic cases.

\subsubsection*{Quasismoothness at all proportional strata}
The proof of Lemma~\ref{lem!wfr} also gives us a slight variation.
\begin{lem}\label{lem!wfrh}
Suppose there exists a subset $I=\{i_1,\dots,i_r\}\subset\{1,\dots,s-1\}$ and $\alpha>0$ such that
$u_i=\alpha k_i$ for all $i\in I$. Set $d:=u_{s+1}$ and $e:=k_{s+1}$. If for some $\lambda>0$ and $h\notin I$ there is a monomial $m=x_hx_{i_1}^{p_1}\cdots x_{i_r}^{p_r}\in F_\lambda$ for which $p_1 k_{i_1}+\cdots+p_rk_{i_r}=e-k_h$ then
\begin{enumerate}
\item\label{lem2i}
$p_1u_{i_1}+\cdots+p_ru_{i_r}=d-u_h$, and
\item\label{lem2ii}
$m\in F_\lambda$ for all $\lambda\ge 0$.
\end{enumerate}
\end{lem}

Suppose the stratum $\Gamma_I=\P(a_{i_1},\dots, a_{i_r})$ is contained in the general $X$. Any monomial $m\in F_\lambda$ of the form in Lemma~\ref{lem!wfrh} gives a tangent variable along the stratum $\Gamma_I$. To be quasismooth along $\Gamma_I$, there must exist at least $r$ such monomials with distinct linear forms~$x_h$.

To use this as a test, we consider each $h\notin I$ in turn, positing a monomial $m\in F_\lambda$ as in the lemma. If some $m$ also satisfies $p_1k_{i_1}+\cdots+p_rk_{i_r}=e-k_h$, then $x_h$ provides one of the tangent forms for every $\lambda$ and we record this fact. Otherwise we may rearrange to obtain an upper bound for $\lambda$ with $x_h$ a tangent form:
\begin{align*}
\lambda=\frac{d-u_h-\sum p_iu_i}{\sum p_ik_i+k_h-e}
&=\frac{(d-u_h+\alpha k_h-\alpha e)-\alpha k_h+\alpha e-\alpha\sum p_i\alpha k_i}{\sum p_ik_i+k_h-e} \\
&=\frac{d-u_h+\alpha k_h-\alpha e}{\sum p_ik_i+k_h-e}-\alpha\\
&\le\abs{(d-\alpha e)-(u_h-\alpha k_h)}-\alpha
\end{align*}
which again provides an upper bound for $\lambda$ in terms of the series. If we find $r$ independent tangent forms along $\Gamma_I$ then it can be contained inside a quasismooth~$X$; if not, then these bounds apply to limit the number of quasismooth members of the series.

\subsection{Analysis of singularities}
Since $X_d\subset\P(a_1,\dots,a_s)$ is general, the quotient singularities of the hypersurfaces can be described by following~\cite[\S10]{IF}. We then apply the Reid--Shepherd-Barron--Tai criterion as given in~\cite[(3.1)]{C3f} and~\cite[Theorem~3.3]{tai}.

The orbifold strata on the ambient $\P(a_1,\dots,a_s)$ correspond to subsets $I\subset \{1,\dots,s\}$ of indices for which $r_I:=\gcd{a_i \mid i \in I}>1$. We only need to work with maximal $I$ for any given $r=r_I$, and so we always assume this is the case. For example, any point on the relative big torus $\Pi^\circ\subset\Pi$ of the $I=\{4,5\}$ stratum $\Pi_I\subset\P(1,3,5,8,12)$ has a quotient singularity
of type $\frac{1}{4}(1,3,5)=\frac{1}{4}(1,3,1)$ transverse to $\Pi_I$.

Given a hypersurface $X_d\colon (F=0)\subset\P(a_1,\dots,a_s)$, consider an ambient orbifold stratum $\Pi=\P(a_{i_1},\dots,a_{i_t})$ of transverse type $\frac{1}{r}(b_1,\dots,b_{s-t})$. So $\{a_{i_1},\dots,a_{i_t},b_1,\dots,b_{s-t}\} = \{a_1,\dots,a_s\}$. One of two things can happen:
\begin{enumerate}
\item
$F$ vanishes on $\Pi$, so that $\Pi\subset X$. In this case, at any point $P\in X\cap\Pi^\circ$, $X$ has a transverse quotient singularity of type $\frac{1}{r}(b_1,\dots,\widehat{b_j},\dots,b_{s-t})$, where $x_j$ is a tangent variable to $X$ at $P$. Note that there may be several tangent variables at $P$, but their weights are congruent modulo~$r$ and so any one may be used.
\item
$F=0$ cuts a codimension~one locus transversely inside $\Pi$. In this case, at any point $P\in X\cap\Pi^\circ$, the hypersurface $X$ has a transverse quotient singularity of type $\frac{1}{r}(b_1,\dots,b_{s-t})$.
\end{enumerate}
This is enough to calculate the singularities of $X$.

\begin{eg}
Consider $X_{112}\subset\P(28,24,21,16,13,11)$ in coordinates $x$, $y$, $z$, $u$, $v$, and $w$. The ambient space has the following orbifold strata:
\begin{description}
\item[$0$-dimensional strata]
$\frac{1}{28}(24,21,16,13,11)$, $\frac{1}{24}(4,21,16,13,11)$, $\frac{1}{21}(7,3,16,13,11)$,\\
$\frac{1}{16}(12,8,5,13,11)$, $\frac{1}{13}(4,11,8,3,11)$, and $\frac{1}{11}(6,2,10,5,2)$.
\vspace{0.5em}
\item[$1$-dimensional strata]
Transverse $\frac{1}{8}(28,21,13,11)=\frac{1}{8}(4,5,5,3)$ along the relatively open stratum in $\P(24,16)$; $\frac{1}{7}(3,2,6,4)$ along $\P(28,21)$; and $\frac{1}{3}(1,1,1,2)$ along $\P(24,21)$.
\vspace{0.5em}
\item[$2$-dimensional strata]
Transverse $\frac{1}{4}(21,13,11)=\frac{1}{4}(1,1,3)$ along $\P(28,24,16)$.
\end{description}
Since $X$ is general it intersects the open two-stratum transversely in a curve of transverse $\frac{1}{4}(1,1,3)$ singularities. It also contains the $\P(24,21)$ stratum: the monomials $y^4u$ and $z^4x$ provide tangent forms along it, so $X$ has transverse type $\frac{1}{3}(1,1,2)$ in the open stratum. These two curves meet at the $y$-coordinate point $P_2$, where $y^4u$ eliminates $u$, so is a dissident point $\frac{1}{24}(4,21,13,11)$. It does not pass through the $0$-strata of indices $28$ and $16$, since there exist pure power monomials $x^4$ and $u^7$. The remaining $0$-strata do lie on $X$: the monomials $z^4x$, $v^7z$ and $w^8y$ (or $w^9v$) provide orbifold tangent equations at those points, which are therefore isolated terminal quotient singularities: $\frac{1}{21}(3,16,13,11)$, $\frac{1}{13}(2,11,3,11)$, and $\frac{1}{11}(6,10,5,2)$. 

The closure of the index~four curve on $X$ has two more dissident points where it meets $\P(24,16)$: the intersection is cut out by $(y^4+u^6)u$, which in the relatively open stratum is two points each of type $\frac{1}{8}(4,5,5,3)$. Finally, $X$ meets the open stratum of $\P(28,21)$ in a single point of type $\frac{1}{7}(3,2,6,4)$ cut out by $(x^3+z^4)x$.

The monomials seen so far already describe terminal singularities, so define a Fano fourfold
\[
X\colon(x^4+y^4u+z^4x+u^7+v^7z+w^8y=0)\subset\P(28,24,21,16,13,11).
\]
\end{eg}

\section{Classifications of hypersurfaces}\label{sec!results}
The main results of this paper concern fourfolds. Before discussing those, we recover known classifications in lower dimensions as a means of checking our implementation. There are also new results in these lower dimensions, but that is not what we focus on. The results in dimensions two, three, and four are summarised in Tables~\ref{tab!2d},~\ref{tab!3d}, and~\ref{tab!4d}, respectively; details are available on the Graded Ring Database~\cite{grdb}. The varieties are polarised by $A\in\abs{\cO_X(1)}$, and have $\omega_X = \cO_X(k)$ for $k=d-\sum a_i$. In particular, $K_X = kA$ and $X$ is not embedded by $\pm K_X$ unless $k=\pm 1$. When $k<0$ we say that $X$ has \emph{index $-k$}. We list only {\em nondegenerate} hypersurfaces, that is, those whose defining equation degree is not equal to one of the weights; degenerate hypersurfaces only arise in the anticanonical case when the index is bigger than the dimension.

\subsection{Two-dimensional orbifold hypersurfaces}
We summarise the results in Table~\ref{tab!2d}. The `famous $95$' weighted K3 hypersurfaces of Reid~\cite{C3f} is the first important result. The $62$ cases of canonically polarised surfaces are not so familiar, but
the four cases of these that are smooth are well known:
\begin{equation}\label{eq!c2f}
X_5\subset\P^3,\quad
X_6\subset\P(2,1,1,1),\quad
X_8\subset\P(4,1,1,1),\quad\text{ and }\quad
X_{10}\subset\P(5,2,1,1).
\end{equation}
The anticanonically polarised surfaces are the result of Johnson--Koll\'ar~\cite[Theorem~8]{JK}: the result is $22$ sporadic cases and a single infinite one-parameter series. Higher index del~Pezzo surfaces have been studied by Boyer, Galicki and Nakamaye~\cite{BGN}, Cheltsov and Shramov~\cite{CS}, and Paemurru~\cite{paemurru}.

\begin{table}[ht]
\caption{Summary of results for surfaces, including the number of canonical, terminal, and smooth cases that occur among the series and sporadic results.}\label{tab!2d}
\centering
\setlength{\extrarowheight}{0.2em}
\begin{tabular}{ccccccccc}
\toprule
$\dim$&$k$&\#series&\#sporadic&\#can&\#term&\#sm&Ref\\
\cmidrule(lr){1-2}\cmidrule(lr){3-7}\cmidrule(lr){8-8}
\oddrow $2$&$-2$&$9$&$\phantom{{}^\dagger}32^\dagger$&$4$&$1$&$1$&\cite{BGN,CS,paemurru}\\
\evnrow $2$&$-1$&$1$&$22$&$3$&$3$&$3$&\cite{JK}\\
\oddrow $2$&$0$&$0$&$95$&$95$&$2$&$2$&\cite{C3f}\\
\evnrow $2$&$1$&$0$&$62$&$4$&$4$&$4$&\\
\oddrow $2$&$2$&$0$&$205$&$8$&$2$&$2$&\\
\evnrow $2$&$3$&$0$&$103$&$11$&$6$&$6$&\\
\oddrow $2$&$4$&$0$&$276$&$11$&$2$&$2$&\\
\evnrow $2$&$5$&$0$&$96$&$11$&$7$&$7$&\\
\bottomrule
\end{tabular}
\parbox{0.75\textwidth}{\footnotesize${}^\dagger$A $33$rd case $X_{18}\subset\P(7,6,4,3)$ lies in a series but with the weights in a different order.}
\end{table}

We consider the case $\omega_X=\cO_X(-2)$ in detail to illustrate the need for careful post-processing of the output from this algorithm when there are infinite series. First, it can happen that sporadic elements also appear in infinite series. Second, it is also typical -- as already happens in~\cite{JK} -- that infinite series are not saturated in their natural integral parameter: for example, often only alternate cases are well-formed. The algorithm as described is unable to analyse this, but we do for Theorem~\ref{thm:index_2_dim_2} below, which is sharp. Whilst the algorithm returns nine series and $37$ sporadic cases, observation (with or without a computer) trims this to a sharp result of nine precisely-specified series and $32$ sporadic cases.

\begin{thm}[Index~two del~Pezzo hypersurfaces]\label{thm:index_2_dim_2}
$\phantom{x}$
\begin{enumerate}
\item\label{th3i}
The general member of each of the cases listed in Table~\ref{tab!dp2} is a well-formed del~Pezzo surface of index~two with log terminal singularities.
\item\label{th3ii}
Conversely, if $X_d\subset\P(a_1,a_2,a_3,a_4)$ is a well-formed del~Pezzo surface of index~two (so that $a_1+\cdots+a_4= d+ 2$) with log terminal singularities then, possibly after reordering the $a_i$, it is one of the cases listed in Table~\ref{tab!dp2}.
\item\label{th3iii}
Of the surfaces in Table~\ref{tab!dp2}, only  $X_2\subset\P^3$, $X_3\subset\P(2,1,1,1)$, $X_4\subset\P(2,2,1,1)$, and $X_6\subset\P(3,2,2,1)$ have a member with canonical singularities, of which only the first has nonsingular member.
\end{enumerate}
\end{thm}

\begin{table}[ht]
\caption{All well-formed index two del~Pezzo hypersurfaces with log terminal singularities.}\label{tab!dp2}
\centering
\centerframebox{
\textbf{One two-dimensional series}
\[
\begin{array}{l@{.\hspace{0.75em}}l@{\hspace{0.75em}}l}
\hypertarget{series:0}{0}&
X_{2+2m+n}\subset\P((1,1,1,1)+m(1,1,0,0) + n(1,0,0,0))&
\ForAll n,m\geq 0.
\end{array}
\]

\textbf{Eight one-dimensional series}
\[
\begin{array}{l@{.\hspace{0.75em}}l@{\ \subset\ }r@{\,+\,n}l@{\hspace{0.75em}}l}
\hypertarget{series:1}{1}&
X_{6+2n}&\P((3,2,2,1)&(1,1,0,0))&
\ForAll n\ge 0.\\
\hypertarget{series:2}{2}&
X_{20+4n}&\P((8,5,5,4)&(2,1,1,0))&
\ForAllEven n\ge 0.\\
\hypertarget{series:3}{3}&
X_{24+6n}&\P((10,8,4,4)&(3,2,1,0))&
\ForAllOdd n\ge -1.\\
\hypertarget{series:4}{4}&
X_{17+4n}&\P((7,5,4,3)&(2,1,1,0))&
\ForAll n\ge 0\text{ congruent to $0\mod 3$.}\\
\hypertarget{series:5}{5}&
X_{21+6n}&\P((9,7,4,3)&(3,2,1,0))&
\ForAll n\ge 0\text{ congruent to $0\mod 3$.}\\
\hypertarget{series:6}{6}&
X_{24+6n}&\P((12,7,4,3)&(3,2,1,0))&
\ForAll n\ge 0\text{ congruent to $0\mod 3$.}\\
\hypertarget{series:7}{7}&
X_{9+2n}&\P((4,3,3,1)&(1,1,0,0))&
\ForAll n\ge 0\text{ congruent to $0$ or $1\mod 3$.}\\
\hypertarget{series:8}{8}&
X_{12+3n}&\P((4,4,3,3)&(1,1,1,0))&
\ForAll n\ge 0\text{ congruent to $0$ or $1\mod 3$.}
\end{array}
\]

\textbf{32 sporadic cases}
\[
\begin{array}{lll}
X_{12}\subset\P(6,4,3,1)&
X_{36}\subset\P(18,12,7,1)&
X_{99}\subset\P(41,29,17,14)\\
X_{12}\subset\P(5,4,3,2)&
X_{36}\subset\P(13,10,9,6)&
X_{105}\subset\P(43,35,19,10)\\
X_{14}\subset\P(7,4,3,2)&
X_{40}\subset\P(20,13,8,1)&
X_{105}\subset\P(47,28,21,11)\\
X_{15}\subset\P(7,5,4,1)&
X_{45}\subset\P(22,15,9,1)&
X_{107}\subset\P(41,32,25,11)\\
X_{16}\subset\P(8,5,4,1)&
X_{48}\subset\P(16,13,12,9)&
X_{107}\subset\P(47,29,20,13)\\
X_{18}\subset\P(9,6,4,1)&
X_{57}\subset\P(22,19,13,5)&
X_{111}\subset\P(43,34,25,11)\\
X_{20}\subset\P(10,5,4,3)&
X_{57}\subset\P(25,19,8,7)&
X_{111}\subset\P(49,31,20,13)\\
X_{22}\subset\P(11,7,5,1)&
X_{57}\subset\P(19,19,12,9)&
X_{135}\subset\P(61,45,18,13)\\
X_{27}\subset\P(13,9,6,1)&
X_{64}\subset\P(32,19,8,7)&
X_{226}\subset\P(113,61,43,11)\\
X_{30}\subset\P(15,10,6,1)&
X_{70}\subset\P(35,19,13,5)&
X_{226}\subset\P(113,71,31,13)\\
X_{30}\subset\P(15,10,4,3)&
X_{81}\subset\P(31,24,19,9)&\\
\end{array}
\]
}
\end{table}

\begin{rem}
The list of series and sporadic cases in Table~\ref{tab!dp2} is given without repetition. The surface $X_{18}\subset\P(7,6,3,4)$ is the initial case $n=-1$ in Series~\hyperlink{series:3}{3}. In that form, it does not respect the ordering of weights that we imposed at the outset, but it falls naturally in the series and so we record it there. (Other works prefer to list this case separately, precisely because it breaks the ordering.) All other cases in the table have the conventional ordering.

Some of the series could be extended by allowing negative values for $n$ and relaxing the ordering of weights, but in each case this only duplicates a hypersurface appearing in another series (after reordering the weights). These coincidences are:
\begin{enumerate}
\item
Series~\hyperlink{series:1}{1} meets Series~\hyperlink{series:0}{0} at $X_4\subset\P(2,1,2,1)$ when $n=-1$;
\item
Series~\hyperlink{series:2}{2} meets Series~\hyperlink{series:8}{8} at $X_{12}\subset\P(4,3,3,4)$ when $n=-2$;
\item
Series~\hyperlink{series:3}{3} meets Series~\hyperlink{series:0}{0} at $X_6\subset\P(1,2,1,4)$ when $n=-3$;
\item
Series~\hyperlink{series:4}{4} meets Series~\hyperlink{series:0}{0} at $X_5\subset\P(1,2,1,3)$ when $n=-3$;
\item
Series~\hyperlink{series:6}{6} meets Series~\hyperlink{series:0}{0} at $X_6\subset\P(3,1,1,3)$ when $n=-3$;
\item
Series~\hyperlink{series:7}{7} meets Series~\hyperlink{series:0}{0} at $X_5\subset\P(2,1,3,1)$ when $n=-2$;
\item
Series~\hyperlink{series:8}{8} meets Series~\hyperlink{series:1}{1} at $X_6\subset\P(2,2,1,3)$ when $n=-2$.
\end{enumerate}
\end{rem}

\begin{proof}[Proof of Theorem~\ref{thm:index_2_dim_2}]
\eqref{th3i} and~\eqref{th3iii} are routine: working in coordinates $x$, $y$, $z$, and $t$ on $w\P^3$, we check that every hypersurface listed is well-formed and quasismooth, and identify those with canonical singularities.

Series~\hyperlink{series:0}{0} has a member $xy+z^{2+n+2m}+t^{2+n+2m}=0$. This satisfies the conditions, which are open, and so the general member does too. This hypersurface meets potential orbifold strata in quotient singularities $\frac{1}{1+n+m}(1,1)$ and $\frac{1}{1+m}(1,1)$, which are canonical if and only if $n+m\le 1$.

We consider Series~\hyperlink{series:8}{8} according to the residue of $n$ modulo~$3$. If $n=3k$ then
\[
x^2t+y^2z+z^{2k+3}+t^{9+2n}=0
\]
meets the one-dimensional $y$-$z$ orbifold stratum $\P(3(k+1),3)$ in two points and satisfies the conditions; it has a $\frac{1}{3}(1,1)$ singularity, which is not canonical. If $n=3k+1$ then
\[
x^2t+y^2z+xz^{k+2}+t^{9+2n}=0
\]
satisfies the conditions; again is not canonical. The remaining series and the sporadic cases are checked similarly.

\eqref{th3ii} follows from the correct implementation of the algorithm, followed by correct organisation of the output. The first step is to analyse the infinite series, confirming that each one does represent infinitely many quasismooth cases, and that the values of $n$ in Table~\ref{tab!dp2} are the only ones that work. For example, in Series~\hyperlink{series:7}{7}, when $n=3k+2$ the one-dimensional $x$-$z$ orbifold stratum $\P(3(k+2),3)$ is contained in $X$, since $9+2n$ is not congruent to $0$ modulo~$3$, so the hypersurface is not well-formed in this case. Other cases are treated similarly. The second step is then to exclude sporadic cases that lie in series, which is a routine observation.
\end{proof}

\subsection{Three-dimensional orbifold hypersurfaces}
We summarise the results in Table~\ref{tab!3d}. Some of these results are well known: the $7555$ Calabi-Yau threefolds agree with Kreuzer--Skarke~\cite{KS}; when $k=1$ we recover Iano-Fletcher's list~\cite{IF} of $23$ canonical hypersurfaces. The other classifications with $k>0$ are new, to the best of our knowledge. The classifications with $k<0$ and terminal singularities agree with those obtained by Suzuki~\cite{suzuki,BS}.

\begin{table}[ht]
\caption{Summary of results for threefolds, including the number of canonical, terminal, and smooth cases that occur among the series and sporadic results.}\label{tab!3d}
\centering
\setlength{\extrarowheight}{0.2em}
\begin{tabular}{ccccccccc}
\toprule
$\dim$&$k$&\#series&\#sporadic&\#can&\#term&\#sm&Ref\\
\cmidrule(lr){1-2}\cmidrule(lr){3-7}\cmidrule(lr){8-8}
\oddrow $3$&$-2$&$66$&$7084$&$96$&$8$&$3$&\\
\evnrow $3$&$-1$&$25$&$4442$&$95$&$95$&$2$&\cite{IF,JK2,C3f}\\
\oddrow $3$&$0$&$0$&$7555$&$7555$&$4$&$4$&\cite{KS}\\
\evnrow $3$&$1$&$0$&$6448$&$23$&$23$&$2$&\cite{IF}\\
\oddrow $3$&$2$&$0$&$11\,762$&$53$&$17$&$6$&\\
\evnrow $3$&$3$&$0$&$8298$&$76$&$27$&$2$&\\
\oddrow $3$&$4$&$0$&$13\,305$&$110$&$25$&$7$&\\
\evnrow $3$&$5$&$0$&$7007$&$83$&$45$&$3$&\\
\bottomrule
\end{tabular}
\end{table}

When $k=-1$ Johnson and Koll\'ar~\cite[Theorem~2.2]{JK2} classify all well-formed quasismooth threefold hypersurfaces with $\omega_X=\cO_X(-1)$. Our algorithm produces $4450$ sporadic cases and $25$ infinite series. Of these $4450$ sporadic cases, eight lie in the infinite series (requiring the given order on the $a_i$). Removing these leaves $4442$ sporadic cases and $25$ infinite one-parameter series, in agreement with~\cite{JK2}. Note that~\cite{JK2} include a further $23$ infinite series that satisfy all the conditions except for well-formedness, and that these series contain no well-formed cases. A further $14$ sporadic cases lie in the series after relaxing the condition on the order of the $a_i$. We choose not to remove them from the sporadic list; these cases are listed in Table~\ref{tab!extra14}. After checking the singularities, we recover the $95$ cases of well-formed quasismooth terminal Fano threefolds discussed after Theorem~(4.5) in~\cite{C3f}, and in~\cite[Corollary~2.5]{JK2}.

\begin{table}[ht]
\caption{Index~one Fano threefolds that lie in infinite series after reordering their weights.}\label{tab!extra14}
\centering
\begin{tabular}{l@{$\qquad$}l@{$\qquad$}l}
\toprule
$X_6\subset\P(2,2,1,1,1)$&$X_8\subset\P(3,2,2,1,1)$&$X_{10}\subset\P(4,3,2,1,1)$\\
$X_{12}\subset\P(5,3,2,2,1)$&$X_{12}\subset\P(5,4,2,1,1)$&$X_{16}\subset\P(7,4,3,2,1)$\\
$X_{16}\subset\P(7,5,2,2,1)$&$X_{18}\subset\P(8,5,3,2,1)$&$X_{20}\subset\P(9,5,4,2,1)$\\
$X_{22}\subset\P(10,7,3,2,1)$&$X_{24}\subset\P(11,8,3,2,1)$&$X_{26}\subset\P(12,7,5,2,1)$\\
$X_{28}\subset\P(13,9,4,2,1)$&$X_{36}\subset\P(17,12,5,2,1)$&\\
\bottomrule
\end{tabular}
\end{table}

We continue calculating in higher index. We say $X$ is a \emph{Fano threefold with canonical singularities} if it satisfies the conditions for a Fano threefold with `terminal' relaxed to `canonical'.

\begin{thm}[Index~two Fano hypersurfaces]\label{thm:index_2_dim_2}
$\phantom{x}$
\begin{enumerate}
\item
There are $66$ one-parameter series and $7084$ sporadic cases of well-formed quasismooth threefold hypersurfaces with $\omega_X=\cO_X(-2)$.
\item
There are $96$ Fano threefold hypersurfaces of index~two with canonical singularities. With the exception of the cubic $X_3\subset\P^4$, they are all of the form $X_d\subset\P(a_1,a_2,a_3,a_4,2)$, where $S_d\subset\P(a_1,a_2,a_3,a_4)$ is one of the $95$ K3 hypersurfaces.
\end{enumerate}
\end{thm}

As in~\cite{suzuki}, three of these $96$ hypersurfaces are nonsingular: $X_3\subset\P^4$, $X_4\subset\P( 2, 1, 1, 1, 1, 4)$, and $X_6\subset\P( 3, 2, 1, 1, 1, 6)$. A further five have terminal singularities:
\begin{gather*}
X_{10}\subset\P( 5, 3, 2, 1, 1),\quad
X_{18}\subset\P( 9, 5, 3, 2, 1),\quad
X_{22}\subset\P( 11, 7, 3, 2, 1),\\
X_{26}\subset\P( 13, 7, 5, 2, 1)
\qquad\text{ and }\qquad
X_{38}\subset\P( 19, 11, 5, 3, 2).
\end{gather*}

We briefly describe the computer analysis. The raw output consists of $7102$ sporadic cases and $85$ series, $84$ of which are one-dimensional, and one of which is two-dimensional. The two-dimensional series is
\[
X_{3+2m+3n}\subset\P((1,1,1,1,1)+m(1,1,0,0,0)+n(1,1,1,0,0)).
\]
When both $m$ and $n>0$ this is not quasismooth along the one-stratum $x_1,x_2$: $x_3$ is the
only tangent form there. This is really two one-parameter series
\[
X_{3+2n}\subset\P((1,1,1,1,1)+n(1,1,0,0,0))
\qquad\text{ and }\qquad
X_{3+3n}\subset\P((1,1,1,1,1)+n(1,1,1,0,0)),
\]
which are both well-formed and quasismooth for all $n\ge 0$. Of the remaining one-parameter series, $64$ have a regular subset of well-formed quasismooth elements: the distribution is either all elements, every other one, every third, or two out of every three. The remaining $20$ one-parameter series are easily seen to have no well-formed quasismooth elements. It remains to observe which of the sporadic cases lie in these $66$ families and to calculate their singularities.

We can compute higher index, although proving precise statements about the infinite series becomes increasingly difficult as the number of results gets larger. We state the results which have canonical singularities, where the statements can be made precise.

\begin{thm}[Index~three Fano hypersurfaces]
There are $100$ Fano threefold hypersurfaces of index~three with canonical singularities. Of these $100$ cases:
\begin{enumerate}
\item
$95$ are of the form $X_d\subset\P(a_1,a_2,a_3,a_4,3)$, where $S_d\subset\P(a_1,a_2,a_3,a_4)$ is one of the $95$ K3 hypersurfaces, and so have a quasismooth K3 hypersurface elephant;
\item
$X_2\subset\P^4$, $X_3\subset\P(2, 1, 1, 1, 1)$ and $X_4\subset\P(2, 2, 1, 1, 1)$ have a quasismooth K3 elephant in codimension~two;
\item
$X_{21}\subset\P(9, 7, 4, 3, 1)$ and $X_{30}\subset\P(10, 9, 7, 4, 3)$ do not have a quasismooth K3 elephant.
\end{enumerate}
\end{thm} 

\noindent
As in~\cite{suzuki}, of these $100$ hypersurfaces only the quadric is nonsingular, and a further six
have terminal singularities: $X_3\subset\P( 2, 1, 1, 1, 1, 3)$, $X_4\subset\P( 2, 2, 1, 1, 1, 4)$, $X_6\subset\P( 3, 2, 2, 1, 1)$, $X_{12}\subset\P( 5, 4, 3, 2, 1, 12)$, $X_{15}\subset\P( 7, 5, 3, 2, 1)$, and $X_{21}\subset\P(8, 7, 5, 3, 1)$.

\begin{thm}[Index~four Fano hypersurfaces]
There are $78$ Fano threefold hypersurfaces of index~four with canonical singularities. Of these $78$ cases:
\begin{enumerate}
\item
$67$ are of the form $X_d\subset\P(a_1,a_2,a_3,a_4,4)$, where $S_d\subset\P(a_1,a_2,a_3,a_4)$ is one of the $95$ K3 hypersurfaces, and so have a quasismooth K3 hypersurface elephant;
\item
Six cases, $X_3\subset\P( 2, 2, 1, 1, 1, 3 )$, $X_4\subset\P( 2, 2, 2, 1, 1, 4 )$, $X_4\subset\P( 3, 2, 1, 1, 1, 4 )$, $X_5\subset\P( 3, 2, 2, 1, 1, 5 )$, $X_6\subset\P( 3, 3, 2, 1, 1, 6 )$, and $X_6\subset\P( 3, 2, 2, 2, 1, 6 )$ have a quasismooth K3 elephant in codimension~two.
\item
Five cases, $X_{20}\subset\P( 8, 6, 5, 4, 1, 20 )$, $X_{28}\subset\P( 10, 8, 7, 4, 3, 28 )$, $X_{28}\subset\P( 14, 8, 5, 4, 1, 28 )$, $X_{36}\subset\P( 18, 8, 7, 4, 3, 36 )$, and $X_{44}\subset\P( 22, 9, 8, 5, 4, 44 )$ do not have a quasismooth K3 elephant.
\end{enumerate}
\end{thm} 

\noindent
Continuing in this fashion -- listing Fano threefolds with canonical singularities of higher index -- there are $46$ cases in index~five, of which $43$ have hypersurface K3 elephants, whilst $X_4\subset\P(3,2,2,1,1)$, $X_6\subset\P(3,3,2,2,1)$, and $X_6\subset\P(4,3,2,1,1)$ have codimension~two elephants. There are $88$ cases in index~six, of which $69$ have a hypersurface elephant, $17$ have a codimension~two elephant, and two cases -- $X_{24}\subset\P( 9, 8, 6, 5, 2)$ and $X_{24}\subset\P(  11, 9, 8, 6, 5)$ -- have no quasismooth K3 elephant.

\subsection{Fano fourfolds}
Our main results are the cases $k=-1$,~$0$, and~$1$. When $k\ge0$ there are no infinite series and the raw output of the algorithm is ready to use.

\begin{table}[ht]
\caption{Summary of results for fourfolds, including the number of canonical, terminal, and smooth cases that occur among the series and sporadic results.}\label{tab!4d}
\setlength{\extrarowheight}{0.2em}
\begin{tabular}{cccccccc}
\toprule
$\dim$&$k$&\#series&\#sporadic&\#can&\#term&\#sm\\
\cmidrule(lr){1-2}\cmidrule(lr){3-7}
\oddrow $4$&$-2$&$\phantom{{}^\dagger}4151^\dagger$&$\phantom{{}^{\dagger\!\dagger}}2\,088\,986^{\dagger\!\dagger}$&$15\,051$&$2304$&$2$\\
\evnrow $4$&$-1$&$1597$&$1\,233\,322$&$11\,618$&$11\,618$&$4$\\
\oddrow $4$&$0$&$0$&$1\,100\,055$&$1\,100\,055$&$33$&$2$\\
\evnrow $4$&$1$&$0$&$1\,338\,926$&$649$&$649$&$6$\\
\oddrow $4$&$2$&$0$&$2\,337\,581$&$1373$&$504$&$2$\\
\evnrow $4$&$3$&$0$&$1\,318\,278$&$1200$&$636$&$7$\\
\oddrow $4$&$4$&$0$&$2\,258\,837$&$2079$&$596$&$3$\\
\evnrow $4$&$5$&$0$&$1\,291\,194$&$1651$&$1017$&$5$\\
\bottomrule
\end{tabular}
\parbox{0.75\textwidth}{\footnotesize${}^\dagger$Experimental result based on 100 initial terms; series rejected by that test have not been proved to have no well-formed quasismooth members.\newline${}^{\dagger\!\dagger}$Sporadic elements that lie in one of the series have not been removed from this list.}
\end{table}

We describe the case $k=-1$ in more detail in Steps~\hyperlink{step:1}{$1$}--\hyperlink{step:5}{$5$} below. In this case raw output consists of four two-parameter series, $1611$ one-parameter series, and $1\,234\,076$ sporadic cases.

\subsubsection*{Step 1}\hypertarget{step:1}{}
The two-parameter series do not contain any good elements. The first case is the series corresponding to integral points in the polyhedron
\[
(15,10,3,2,1,1,31)+\cone{(15,10,3,2,0,0,30),(15,10,3,2,2,0,32)}.
\]
The cone is not regular, but has three semigroup generators, including $(15,10,3,2,1,0,31)$, so its integral points are all of the form
\[
\left(15N,10N,3N,2N,1+n_2+2n_3,1,30N+1+n_2+2n_3\right),
\]
where $N:=n_1+n_2+n_3+1$ and $n_1,n_2,n_3\ge 0$. These integral points correspond to
\[
X_{30N+1+n_2+2n_3}\subset\P(15N,10N,3N,2N,1+n_2+2n_3,1).
\]
When $N>0$ this is not well-formed. The vertex $(15,10,3,2,1,1,31)$ does work, and corresponds to $X_{31}\subset\P(15,10,3,2,1,1)$, but this is already contained in the sporadic list.

The other three cases are
\begin{align*}
(4,3,3,2,1,1,13)&+\cone{(4,3,3,2,0,0,12),(4,3,3,2,2,0,14)},\\
(4,3,3,2,2,1,14)&+\cone{(4,3,3,2,0,0,12),(4,3,3,2,2,0,14)},\\
\text{ and }\qquad
(15,10,3,2,2,1,32)&+\cone{(15,10,3,2,0,0,30),(15,10,3,2,2,0,32)}.
\end{align*}
They follow the same pattern as above; in each case the hypersurface corresponding to the vertex is already in the sporadic list.

\subsubsection*{Step 2}\hypertarget{step:2}{}
$1597$ of the one-parameter families match those predicted by the Johnson--Koll\'ar conjecture, so they all work. We normalise them so they are the intersection of an affine line with the strict positive quadrant, irrespective of ordering of the vertex.

\subsubsection*{Step 3}\hypertarget{step:3}{}
The remaining $14$ one-parameter series, listed in Table~\ref{tab!1d14}, do not contain infinitely many good members.
\begin{table}[ht]
\caption{Fourteen 1-parameter series excluded at Step~3.}\label{tab!1d14}
\setlength{\extrarowheight}{0.2em}
\begin{tabular}{rcl}
\toprule
\oddrow $(105,53,30,21,2,2,212)$&$+$&$\cone{(106,53,30,21,2,0,212)}$\\
\evnrow $(109,55,27,24,4,2,220)$&$+$&$\cone{(110,55,27,24,4,0,220)}$\\
\oddrow $(117,59,33,21,5,2,236)$&$+$&$\cone{(118,59,33,21,5,0,236)}$\\
\evnrow $(117,59,39,18,2,2,236)$&$+$&$\cone{(118,59,39,18,2,0,236)}$\\
\oddrow $(129,65,36,21,8,2,260)$&$+$&$\cone{(130,65,36,21,8,0,260)}$\\
\evnrow $(157,79,63,15,1,2,316)$&$+$&$\cone{(158,79,63,15,1,0,316)}$\\
\oddrow $(165,83,66,15,2,2,332)$&$+$&$\cone{(166,83,66,15,2,0,332)}$\\
\evnrow $(181,91,72,15,4,2,364)$&$+$&$\cone{(182,91,72,15,4,0,364)}$\\
\oddrow $(193,97,45,28,24,2,388)$&$+$&$\cone{(194,97,45,28,24,0,388)}$\\
\evnrow $(213,107,84,15,8,2,428)$&$+$&$\cone{(214,107,84,15,8,0,428)}$\\
\oddrow $(217,109,40,36,33,2,436)$&$+$&$\cone{(218,109,40,36,33,0,436)}$\\
\evnrow $(217,109,69,22,18,2,436)$&$+$&$\cone{(218,109,69,22,18,0,436)}$\\
\oddrow $(277,139,108,16,15,2,556)$&$+$&$\cone{(278,139,108,16,15,0,556)}$\\
\evnrow $(301,151,117,19,15,2,604)$&$+$&$\cone{(302,151,117,19,15,0,604)}$\\
\bottomrule
\end{tabular}
\end{table}
For example, consider the series
\[
(213, 107, 84, 15, 8, 2, 428)+\cone{(214, 107, 84, 15, 8, 0, 428)},
\]
expressed as
\[
X_{428(n+1)}\colon (F=0) \subset \P(214(n+1) - 1, 107(n+1), 84(n+1), 15(n+1), 8(n+1), 2).
\]
The $x_3$-$x_4$ stratum $\Gamma$ is contained in any $X$ (by mod~$3$ and mod~$9$ congruence). But we see that $x_1$, $x_2$, and $x_6$ cannot be tangent forms along $\Gamma$ unless $n=0$. For example, if $x_3^ax_4^bx_6\in F$ then $84a + 15b + 2/(n+1) = 428$, hence $n=0$ or $1$; and since $b$ must be even, $n\neq 1$. Similarly $x_3^ax_4^bx_1\notin F$. If $x_3^ax_4^bx_2\in F$ then, again by computing degrees and dividing by $2(n+1)$, we have that $28a + 5b = 107$, but this has no solutions with $a,b\in\N$. We conclude that $x_5$ is the only tangent form along $\Gamma$, and so $X$ is not quasismooth unless $n=0$. We add this initial $n=0$ element to the sporadic list and exclude the series.
The remaining cases work similarly: in each case we need to record the initial element but no others. 

At the end of this step, we have excluded the two-parameter series and $14$ of the one-parameter series, and have added $14$ additional elements to the sporadic list.

\subsubsection*{Step 4}\hypertarget{step:4}{}
Remove any members of the sporadic list that lie in the series. The form of the Johnson--Koll\'ar conjecture helps: rather than testing membership, one can simple observe when a sporadic case is of the right form; this also finds sporadic elements who lie in a series after re-ordering. There are $768$ sporadic elements that lie in the series, of which $381$ only do so after re-ordering; we remove these cases from the sporadic list.

\subsubsection*{Step 5}\hypertarget{step:5}{}
Check for canonical singularities in the sporadic list and the series. Again, the form of the Johnson--Koll\'ar conjecture helps: every member of a series has the form
\[
X_{2h\sum b_i}\subset\P(-1+h\sum b_i, hb_1,\dots,hb_{s-2}, 2),
\]
where $n\ge 3$ and $h\ge 1$ is odd. Whenever $h>1$, the $(s-3)$-stratum $\P(hb_1,\dots,hb_{s-2})$ has non-trivial stabiliser $\Z/h$, and $X$ intersects this in codimension~two on~$X$. Its transverse quotient type is $\frac{1}{h}(h-1,2)$, which is not canonical when $h\ge3$ and so, in particular, $X$ does not have canonical singularities when $h>1$. Hence we need only check the sporadic cases and the first member of each series.

\subsection{Higher index Fano fourfolds: experimental results}
There is no known sharp upper bound for the index of Fano fourfold hypersurfaces. For Fano threefolds, Suzuki~\cite[Theorem 0.3]{suzuki} proves that the highest Fano index is realised by a weighted projective space, and Prokhorov~\cite[Theorem 1.4]{prokhorov} proves that only weighted projective space achieves this. Taking this as a guide, one may expect the highest index of a ($\Q$-factorial terminal) Fano weighted projective space to be an upper bound. The classification of all such $\P(a_1,\dots,a_5)$ is known~\cite[Theorem 3.5]{kasprzyk}: there are $28\,686$ cases in total, and the highest index is $881$, realised by $\P(430,287,123,21,20)$. Although we have the raw results in a few higher indices, we only discuss the case of index~two here, where we give an experimental overview, as we explain below. The raw output in this case consists of $5799$ series and of $2\,088\,986$ sporadic cases.

For all results so far we have taken the raw output from the algorithm and analysed it in detail, proving non-existence individually for those series we reject. In fact, those series that worked always did so because there is some small periodicity in the behaviour of their elements: typically every other element represents a well-formed quasismooth hypersurface. This has given us computer-assisted rigorous proofs. \emph{We now adopt a more experimental approach}: in what follows we reject any series that does not have a well-formed quasismooth element in its first $100$ terms (in practice we have experimented with increasing the number of terms considered, and have found no difference in outcome).

There are five two-dimensional series, but these do not contribute anything other than existing one-dimensional series or sporadic results, so we reject them. Among the remaining $5794$ one-dimensional series, $1643$ have no good cases among the first $100$ elements, and so we reject these too. This leaves $4151$ one-dimensional series. These fall into several different types. The largest is a natural generalisation of the Johnson--Koll\'ar series, but with $a_6=4$:
\begin{align*}
X_{d^{(k)}}\subset\P(a_1^{(k)},\dots,a_6^{(k)})&\qquad\text{ for even }k\ge 0,\text{ where }\\
(a_1^{(k)},\dots,a_6^{(k)},d^{(k)})&=\left(\sum b_i-2,b_1,\dots,b_4,4,2\sum b_i\right)+k\left(\sum b_i,b_1,\dots,b_4,0,2\sum b_i\right)
\end{align*}
for each of the $2390$ well-formed quasismooth Calabi--Yau threefolds $S_{2\sum b_i}\subset\P(\sum b_i,b_1,\dots,b_4)$. Note that, in contrast to the Johnson--Koll\'ar correspondence in index~one, we do not restrict to those $1597$ cases with well-formed $\P(b_1,\dots,b_4)$.

Another large type is similar: there are $1585$ cases with $a_6=3$ of the form
\[
(a_1^{(k)},\dots,a_6^{(k)},d^{(k)})=\left(b_1-1,b_2,b_3,b_4,b_5,3,\sum b_i\right)+k\left(b_1,b_2,b_3,b_4,b_5,0,\sum b_i\right)
\]
where $X_{\sum b_i}\subset\P(b_1,b_2,b_3,b_4,b_5)$ is a Calabi--Yau threefold hypersurface
that is a triple cover $\sum b_i = 3b_1$ (in some cases after reordering so that $b_1$ is no longer the biggest of the $b_i$), and $k$ is constrained modulo~$3$.

The remaining cases also form patterns:
\begin{description}
\item[3 cases] ($a_5=a_6=1$)
For all $k\ge0$
\[
( a_1^{(k)}, \dots, a_6^{(k)}, d^{(k)}) =
\left(b_1,b_2,b_3,1,1,1,\sum b_i+1\right) + k\left(b_1,b_2,b_3,0,0,0,\sum b_i\right)
\]
where $(b_1,b_2,b_3) = (1,1,1)$, $(2,1,1)$ or $(3,2,1)$.
\item[95 cases]  ($a_5=a_6=1$)
For all $k\ge0$
\[
( a_1^{(k)}, \dots, a_6^{(k)}, d^{(k)}) =
\left(b_1,b_2,b_3,b_4,1,1,\sum b_i\right) + k\left(b_1,b_2,b_3,b_4,0,0,\sum b_i\right)
\]
where $X_{\sum b_i}\subset\P(b_1,b_2,b_3,b_4)$ is any of the $95$ K3 hypersurfaces.
\item[48 cases]  ($a_5=2$, $a_6=1$) For all $k\ge0$
\[
( a_1^{(k)}, \dots, a_6^{(k)}, d^{(k)}) =
\left(b_1-1,b_2,b_3,b_4,2,1,\sum b_i\right) + k\left(b_1,b_2,b_3,b_4,0,0,\sum b_i\right)
\]
where $X_{\sum b_i}\subset\P(b_1,b_2,b_3,b_4)$ is any of the $48$ K3 hypersurfaces that are double covers. A further \textbf{8 cases} arise from these K3 surfaces, also with $a_5=2$, $a_6=1$, but with more adjustments to the first three entries $b_1, b_2, b_3$ of the initial term.
\item[22 cases] ($a_6=3$)
For  $k$ subject to a condition modulo~$3$, with kernels similarly determined by certain Calabi--Yau threefolds, but where the weights of the initial term of the series is half the kernel after two entries are slightly modified. For example:
\begin{align*}
(15,6,5,3,3,3,33)&+k(33,11,10,6,6,0,66)
\quad\text{ for }k\equiv 2\mod 3,\\
\text{ and }\quad
(9,5,4,1,1,3,63)&+k(21,10,7,2,2,0,42)
\quad\text{ for }k\equiv 1\text{ or }2\mod 3.
\end{align*}
\end{description}
We do not see a single systematic pattern to this classification of infinite series, but it does still seem to be determined by Calabi--Yau hypersurfaces of lower dimension.

At this stage it would also be possible also to run through the sporadic cases to determine which lie in one of the series. We have not done this, so the number we give for the sporadic classification is an upper bound.

\subsection{Higher dimensions}
\subsubsection*{Nonsingular weighted hypersurfaces}
If $X_d\subset\P(a_1,\dots,a_s)$ is nonsingular then
\[
Y_d^{(n)}\subset\P(a_1,\dots,a_s,\stackrel{n}{\overbrace{1,\dots,1}})
\]
is a nonsingular $(n+s-2)$-fold for any $n\ge0$. The converse holds too: smoothness for any $n\ge0$ implies smoothness for them all. Whilst it is easy to make nonsingular hypersurfaces of general type with empty canonical system -- for example, $X_{30}\subset\P(5,3,2)$ -- that necessarily involves high degree equations, $d=a_1\cdots a_s$ at least. So it is not a surprise that the numbers of smooth cases in dimensions~three and~four in Tables~\ref{tab!3d} and~\ref{tab!4d} match exactly the numbers of smooth surfaces in Table~\ref{tab!2d}. We list the first cases in Table~\ref{tab!sm}, presented as (not necessarily well-formed) orbifold zero-dimensional schemes $Z_d\subset\P(a_1,a_2)$.

\begin{table}[ht]
\caption{The orbifold zero-dimensional schemes $Z_d\subset\P(a_1,a_2)$ of index $I\leq 10$}\label{tab!sm}
\centering
\setlength{\extrarowheight}{0.2em}
\begin{tabular}{ccl}
\toprule
$I$&$\#$&$(a_1,a_2,d)$ of index $I = d-a_1-a_2$\\
\cmidrule(lr){1-2}\cmidrule(lr){3-3}
\oddrow $0$&$1$&$(1,1,2)$\\
\evnrow $1$&$3$&$(1,1,3),  (2,1,4),  (3,2,6)$\\
\oddrow $2$&$2$&$(1,1,4),  (3,1,6)$\\
\evnrow $3$&$4$&$(1,1,5),  (4,1,8),  (2,1,6),  (5,2,10)$\\
\oddrow $4$&$2$&$(1,1,6),  (5,1,10)$\\
\evnrow $5$&$6$&$(1,1,7),  (6,1,12),  (2,1,8),  (3,1,9),  (4,3,12),  (7,2,14)$\\
\oddrow $6$&$2$&$(1,1,8),  (7,1,14)$\\
\evnrow $7$&$7$&$(1,1,9),  (8,1,16),  (2,1,10),  (4,1,12),  (3,2,12),  (5,3,15),  (9,2,18)$\\
\oddrow $8$&$3$&$(1,1,10),  (9,1,18),  (3,1,12)$\\
\evnrow $9$&$5$&$(1,1,11),  (10,1,20),  (2,1,12),  (5,1,15),  (11,2,22)$\\
\oddrow $10$&$2$&$(1,1,12),  (11,1,22)$\\
\bottomrule
\end{tabular}
\end{table}

\subsubsection*{Series of Fano five-folds}
Fano five-folds that have quasismooth Calabi--Yau fourfold elephants arise directly from the $1\,100\,055$ hypersurfaces: simply include an additional~one amongst the weights.  But we can also use Calabi--Yau fourfolds to describe infinite series of anticanonically-polarised five-fold hypersurfaces using Johnson and Koll\'ar's construction: $360\,346$ of the Calabi--Yau hypersurfaces are double covers $V_{2\sum b_i}\subset\P(\sum b_i,b_1,\dots,b_5)$, and of these $261\,195$ have $\P(b_1,\dots,b_5)$ well-formed. Each of these determine an infinite series according to the Johnson--Koll\'ar recipe of~\S\ref{s!jkconj}. In each series, only the initial element could possibly have terminal singularities, and $31\,400$ Fano five-folds of the form $X_{2\sum b_i}\subset\P(-1+\sum b_i,b_1,\dots,b_5,2)$ arise in this way. In fact $25\,276$ of these $31\,400$ cases already arise in the usual way by including an additional~one among the weights of a Calabi--Yau fourfold: cases like $X_{10}\subset\P(4,1,1,1,1,1,2)$ can be realised by both approaches.

\subsection{Notes on computation and raw results}
Our implementation of the algorithm can be downloaded from~\cite{grdb}. The results can be downloaded as plain-text files, or may be queried online. The code was run on version $2.21\text{-}1$ of the Magma computer algebra system~\cite{magma}.

\begin{table}[ht]
\caption{Raw data from the algorithm in dimension $2$.}\label{tab!raw2}
\centering
\begin{tabular}{rgwgwgwgwgwg}
\toprule
$k$&$-5$&$-4$&$-3$&$-2$&$-1$&$0$&$1$&$2$&$3$&$4$&$5$\\
\cmidrule{1-12}
\#series & 14  & 17 & 6 &  9  & 1 &  0 &  0 & 0 &  0  & 0 & 0  \\
\#sporadic &  11 & 21 & 14 & 37  & 22 & 95 & 62 & 205 & 103  & 276 & 96 \\
\cmidrule{1-12}
Runtime (s) & 2 & 3 & 2 & 2 & 1 & 0.5 & 1 & 3 & 4 & 8 & 10  \\
\bottomrule
\end{tabular}
\end{table}

\begin{table}[ht]
\caption{Raw data from the algorithm in dimension $3$.}\label{tab!raw3}
\centering
\begin{tabular}{ryxyxyxyxyxy}
\toprule
$k$&$-5$&$-4$&$-3$&$-2$&$-1$&$0$&$1$&$2$&$3$&$4$&$5$\\
\cmidrule{1-12}
\#series & 90 & 133 & 59 & 85 & 25 &  0 &  0 & 0 &  0  & 0  &  0  \\
\#sporadic & 3178 & 6065 & 4354 & 7102 & 4450 &  7555 & 6448 & 11762 &  8298  & 13305  & 7007 \\
\cmidrule{1-12}
Runtime (s)& 289 & 347 & 176 & 149 & 72 & 59 & 171 & 392 & 649 & 1154 & 1225   \\
\bottomrule
\end{tabular}
\end{table}

\begin{table}[ht]
\caption{Raw data from the algorithm in dimension 4.}\label{tab!raw4}
\centering
\begin{tabular}{rhHhHh}
\toprule
$k$&$-5$&$-4$&$-3$&$-2$&$-1$\\
\cmidrule{1-6}
\#series & 7461 & 9108 & 5091 & 5799 & 1615   \\
\#sporadic & 1\,041\,428 & 1\,855\,631 & 1\,129\,239  &  2\,088\,986  &  1\,234\,076     \\
\cmidrule{1-6}
Sing.\ time (s)& 9728 & 16393 & 9032 & 8082 &  3450   \\
\bottomrule
\end{tabular}\hspace{8em}\vspace{0.5em}

\hspace{8em}\begin{tabular}{HhHhHh}
\toprule
$0$&$1$&$2$&$3$&$4$&$5$\\
\cmidrule{1-6}
0 &  0 & 0 &  0  & 0  &  0   \\
1\,100\,055  &  1\,338\,926  &  2\,337\,581   &  1\,318\,278   &  2\,258\,837  &   1\,291\,194     \\
\cmidrule{1-6}
1391  & 4065  & 5335  & 4053 & 5310 & 3720  \\
\bottomrule
\end{tabular}
\end{table}

Calculations in dimensions~two and~three for small $k$ can easily be carried out on a single computer. In Tables~\ref{tab!raw2} and~\ref{tab!raw3} we indicate approximate timings on a standard laptop, along with the size of the raw output. In dimension~four the search tree is too large to be tackled reasonably in a single run. Instead, we run the algorithm to recursive depth~one and compute the possible monomials at the first coordinate point~$P_1$. We then run each of these as a single process on a high-performance computing cluster of about $200$ cores, together with a data management layer to handle results, failed processes, and so on. This runs in the order of an hour. Text files containing the results, and functions to parse them as Magma input, are available at~\cite{grdb}. We record the numbers of raw output and the timings to compute singularities in Table~\ref{tab!raw4}.

\subsection*{Acknowledgments}
We are grateful to Miles Reid and Olof Sisask who explained this algorithm to us in the first place, and whose re-implementation of it in $2005$ to recalculate the $7555$ Calabi--Yau threefolds for the database at~\cite{grdb} was our starting point, and also to Jennifer Johnson and J\'anos Koll\'ar for discussion of their conjecture. Our thanks to John Cannon for providing Magma for use on the Imperial College mathematics cluster and to Andy Thomas for technical support. This work was supported in part by EPSRC grant~EP/E000258/1. AK is supported by EPSRC grant~EP/I008128/1 and ERC Starting Investigator Grant number~240123.
\bibliographystyle{plain}

\end{document}